\documentclass[11pt]{amsart}

\usepackage{amssymb,graphicx, amsmath, amsthm, enumitem,MnSymbol, array,mathalpha}
\usepackage{calrsfs}
\usepackage{wrapfig}
\usepackage{graphicx}
\graphicspath{ {./images/} }

\DeclareMathAlphabet\mathbfcal{OMS}{cmsy}{b}{n}

\usepackage{todonotes}

\usepackage{floatflt}

\usepackage{hyperref}

\usepackage{mathbbol}

 \usepackage[mathscr]{eucal}

\usepackage[font=small]{caption}

\def\k{\mathbb R} 

\def\k{{\kappa}}

\def\PP{f}

\def\PP{\overline{\PP}}

\def\P{{\mathbb{P}}}
\def\mid{{\rm mid}}

\def\B{\mathbb B}

\begin{document}

\newtheorem{theorem}{Theorem}[section]
\newtheorem{lemma}[theorem]{Lemma}
\newtheorem{proposition}[theorem]{Proposition}
\newtheorem{corollary}[theorem]{Corollary}
\newtheorem{problem}[theorem]{Problem}
\newtheorem{construction}[theorem]{Construction}

\theoremstyle{definition}
\newtheorem{defi}[theorem]{Definitions}
\newtheorem{definition}[theorem]{Definition}
\newtheorem{notation}[theorem]{Notation}
\newtheorem{remark}[theorem]{Remark}
\newtheorem{example}[theorem]{Example}
\newtheorem{question}[theorem]{Question}
\newtheorem{comment}[theorem]{Comment}
\newtheorem{comments}[theorem]{Comments}

\newtheorem{discussion}[theorem]{Discussion}

\renewcommand{\thedefi}{}

\long\def\alert#1{\smallskip{\hskip\parindent\vrule%
\vbox{\advance\hsize-2\parindent\hrule\smallskip\parindent.4\parindent%
\narrower\noindent#1\smallskip\hrule}\vrule\hfill}\smallskip}

\def\red{{{\rm red}}}
\def\ff{\frak}
\def\tf{torsion-free}
\def\Spec{\mbox{\rm Spec }}
\def\Proj{\mbox{\rm Proj }}
\def\hgt{\mbox{\rm ht }}
\def\type{\mbox{ type}}
\def\Hom{\mbox{ Hom}}
\def\rank{\mbox{rank}}
\def\Ext{\mbox{ Ext}}
\def\Tor{\mbox{ Tor}}
\def\ker{\mbox{ Ker }}
\def\Max{\mbox{\rm Max}}
\def\End{\mbox{\rm End}}
\def\xpd{\mbox{\rm xpd}}
\def\Ass{\mbox{\rm Ass}}
\def\emdim{\mbox{\rm emdim}}
\def\epd{\mbox{\rm epd}}
\def\repd{\mbox{\rm rpd}}
\def\ord{\mbox{\rm ord}}

\def\Bis{\mbox{\rm Bis}}

\def\htt{\mbox{\rm ht}}

\def\DD{{\mathcal D}}
\def\EE{{\mathcal E}}
\def\FF{{\mathcal F}}
\def\GG{{\mathcal G}}
\def\HH{{\mathcal H}}
\def\II{{\mathcal I}}
\def\LL{{\mathcal L}}
\def\MM{{\mathcal M}}
\def\PP{{\mathsf{P}}}

\def\k{\mathbb{k}}

\title{Bisector fields and projective duality}

\author{Bruce Olberding} 
\address{Department of Mathematical Sciences, New Mexico State University, Las Cruces, NM 88003-8001}

\email{bruce@nmsu.edu}

\author{Elaine A.~Walker}
\address{Las Cruces, NM}

\email{miselaineeous@yahoo.com}

\begin{abstract}   
Working  over a field $\k$ of characteristic $\ne 2$, we study what we call bisector fields, which are    arrangements  of paired lines in the  plane that have the property that each line in the arrangement crosses the paired lines in pairs of points that all share the same midpoint.  To do so, we use  tools from the theory of algebraic curves and projective duality. We obtain a complete classification  if $\k$ is real closed or algebraically closed, and we obtain a partial classification if $\k$ is a finite field. A classification for other fields remains an open question. Ultimately this is a question regarding affine equivalence within a  system of certain rational quartic  curves. 
 \end{abstract}


\thanks{\today}

\maketitle

\section{Introduction} 

Throughout this article, $\k$ denotes a field whose characteristic is not $2$. 
 We work in the affine plane over $\k$ and 
  study   arrangements of paired lines    like those in Figures~1 and~2, 
   where each line   in the arrangement  crosses the paired lines in pairs of points  
that all share the same midpoint.  
We call a line with this property a  {\it bisector} of the arrangement and we call the distinguished midpoint on the line, the midpoint shared by all pairs, the {\it midpoint} of the bisector.  
Although  not obvious, if $\k$ is infinite and the arrangement finite, 
 then infinitely many more pairs of bisectors can be added to this arrangement while preserving the symmetry that 
the paired lines   are both bisectors of, {\it and bisected by}, the other pairs of lines in the  arrangement.
We call an arrangement of paired lines that has this self-bisection property and     cannot be further enlarged  
  a {\it bisector field}. See Section~2 for a formal definition. In \cite[Theorems 4.5~and~6.3]{OWPencil} it is shown that every bisector field is the set of bisectors of a quadrilateral and  is the set of line pairs that occur as degenerations of the pencil of conics through the vertices of this quadrilateral. Further properties of bisector fields of quadrilaterals, some of which we recall in Section~2, are studied in \cite{OWQuad}.


  Our goal in this article is to further describe the nature of bisector fields, including how they arise and their classification up to affine equivalence.
  The methods we use include classical  point/line duality of the projective plane and the  theory of algebraic curves over arbitrary fields,  
but these are mostly on an elementary level. 
We define a boundary of a bisector field $\B$ and show it is  a rational quartic curve, a parabola or a point. We prove in the first two cases
  that all the lines tangent to the boundary of~$\B$ are bisectors in $\B$. In the first case, that 
in which the boundary is a quartic,  
  the set of tangent lines is the entire bisector field, and so the bisector field obeys a kind of holographic principle in which all information in the bisector field is encoded on the boundary. See Figures~1(a) and~1(b). 
   In the second case, that in which the boundary is a parabola as in Figure~1(c), there are additional bisectors that are not tangent to the boundary, but these are easily described as a pencil of parallel lines.  The final case is the least interesting one, that in which the boundary is a point, all the lines through the point are bisectors and there are two additional pencils, each consisting  of parallel lines; see Figure~1(d).
   
    \begin{figure}[h] \label{diagonalspar}
 \begin{center}
 \includegraphics[width=0.95\textwidth,scale=.09]{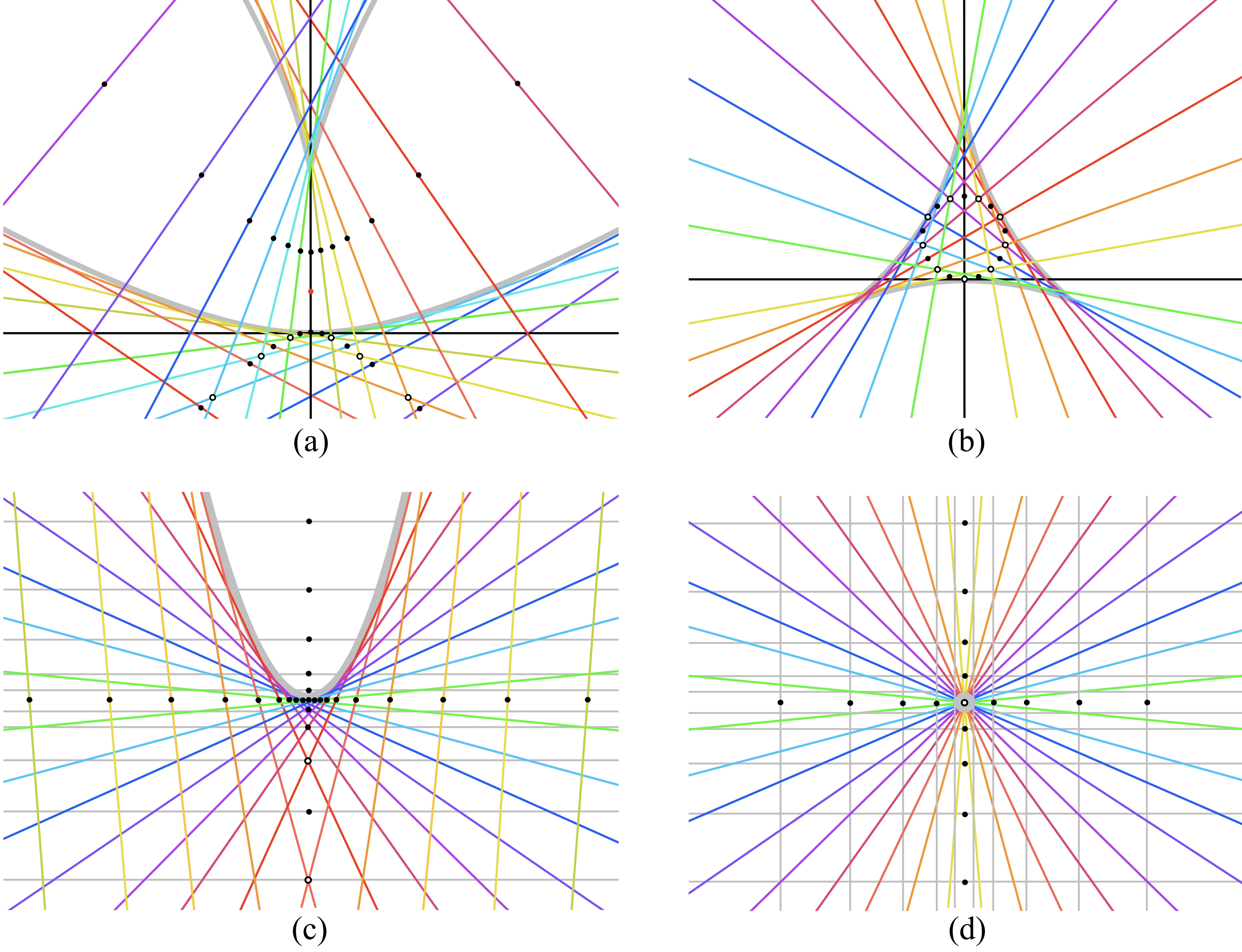} 
 \end{center}
 \caption{
 It is proved in Section~5 that over the field of real numbers, each bisector field is affinely equivalent to one of the four  bisector fields represented here. (Only selected bisectors are shown; the bisector field itself would cover the entire plane.). Two lines of the same color indicate a bisector pair.  
 The midpoints of the bisectors are marked with  black points. White points indicate an intersection of two bisectors from the same pair that is not also a midpoint for any of the bisectors visible in the figures.   
 }
\end{figure}

Using these ideas, as 
well as an arithmetic criterion given in  Section~5 for determining affine equivalence of bisector fields, 
 we classify in the last section of the paper the bisector fields over an algebraically closed field and show in this case that up to affine equivalence there are three bisector fields, each of which can be explicitly described. Over a real closed field, we show there are four bisector fields, and over the field of rational numbers there are infinitely many. See Figure~1 for the classification over the field of real numbers. For finite fields, we are only able to give partial results regarding classification. This raises the question, stated at the end of the paper,  of how many bisectors fields there are for a given choice of the field $\k$.   
 

  
  
  
 




 \begin{figure}[h] \label{diagonalspar}
 \begin{center}
 \includegraphics[width=0.6\textwidth,scale=.09]{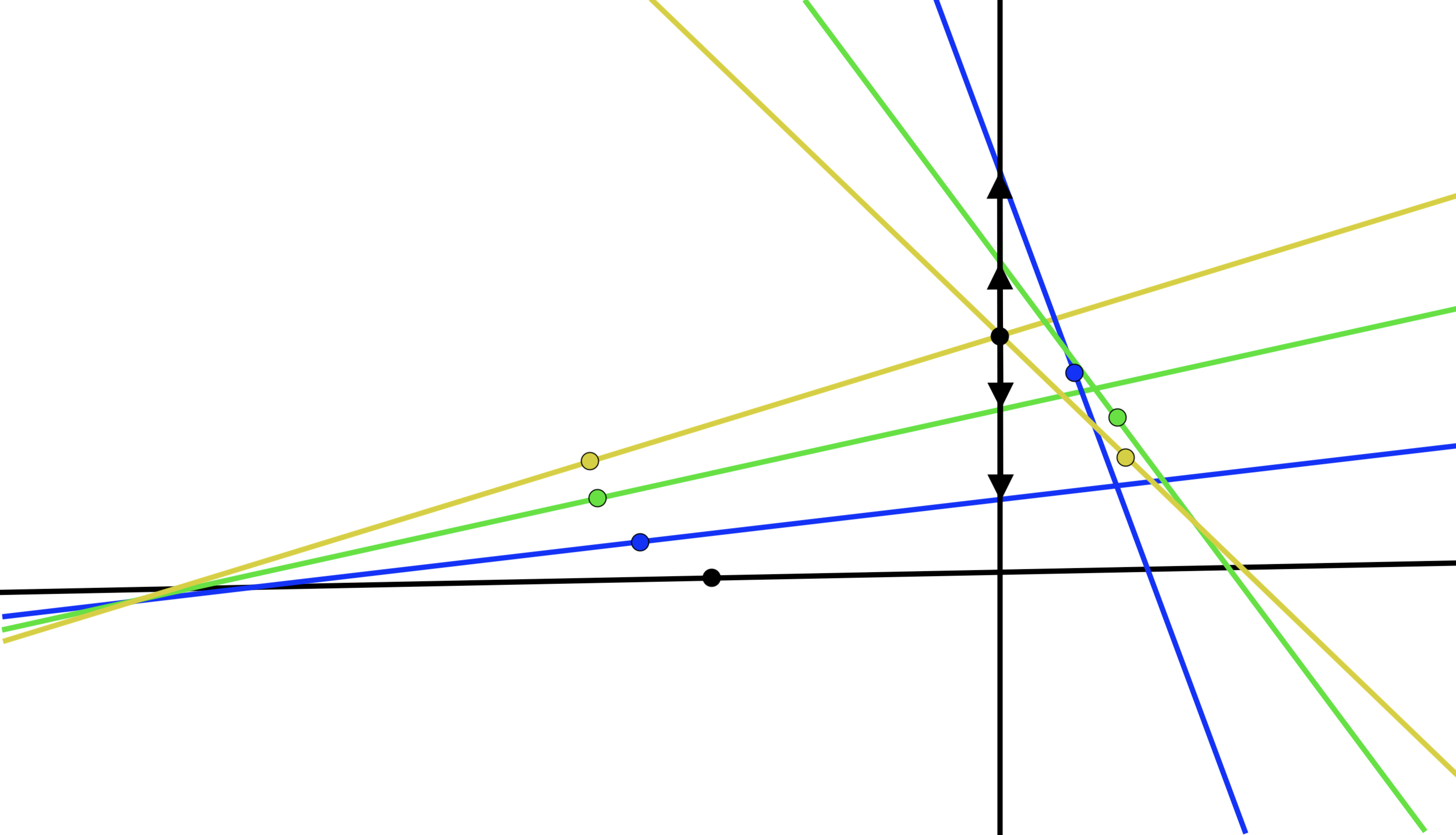} 
 \end{center}
 \caption{
 Four pairs of bisectors from a bisector field. The arrows on the black line help visualize the bisection property. The points shown are the midpoints of the bisectors. Any line chosen from a bisector field along with its midpoint will bisect {\it all} pairs from this same midpoint. 
 }
\end{figure}

{\it Terminology.} By a {\it $($complete$)$ quadrilateral} $Q=ABA'B'$ we mean a collection of four distinct lines $A,B,A',B'$, the {\it sides} of $Q$, and their six points of intersection, some of which are possibly at infinity. 
The pairs  $A,A'$ and $B,B'$ are the pairs of {\it opposite sides} of $Q$. All other pairs are pairs of {\it adjacent sides.} The intersection of 
 two adjacent sides is a {\it vertex} of~$Q$.  The {\it diagonals} of $Q$ are the two lines through non-adjacent vertices  of $Q$. We require that no adjacent pair of adjacent sides of $Q$ consists of parallel lines, and that all four sides do not go through the same point. We do, however, allow the case in which three sides go through a single point. (In \cite{OWPencil}, adjacent sides of a quadrilateral are allowed to be parallel, but we are excluding that case here. In the terminology of \cite{OWPencil}, we are restricting  to {\it affine} quadrilaterals.)

\section{Background: Bisector fields}


%

 We formalize the notion of a bisector field and recall several properties of bisectors from~\cite{OWQuad} in order to give more context for the results of the paper. 
 Apart from the discussion in this section, the  present paper is mostly independent of the articles \cite{OWPencil} and \cite{OWQuad}. 

Let $\PP$ be a pair of lines and $\ell$ be a line. 
Denote by  $\mid_{\PP}(\ell)$  the midpoint of the two points where $\ell$ meets $\PP$, with the following stipulations. 
If exactly one of these points is at infinity,  
 then  $\mid_{\PP}(\ell)$ is  defined to be the point at infinity for $\ell$, while if both are at infinity or $\ell$ is one of the lines in $\PP$, then $\mid_{\PP}(\ell)$ is left undefined. If neither of the two points is at infinity then $\mid_{\PP}(\ell)$  is {\it finite}. 
  If $p$ is a point, the line 
 $\ell$ {\it bisects a set~$\B$ of pairs of lines with midpoint $p$}    if $p = \mid_{\PP}(\ell)$ for all  pairs $\PP$ in $\B$ such  that $ \mid_{\PP}(\ell)$ is defined.


 \begin{definition} \label{bf} A set $\B$  of pairs of lines is an {\it affine bisector arrangement} if each line in each pair in $\B$  bisects $\B$ with a finite midpoint (see Figure~2). An {\it affine bisector field} is  an affine bisector arrangement $\B$ 
 such that not all lines in $\B$ go through the same point and $\B$ 
 cannot be extended to a larger bisector arrangement. 
The pairs of lines in $\B$ are referred to as {\it $\B$-pairs.} 
\end{definition} 

See Figure~1 for examples of affine bisector fields. 
 The reason for the adjective ``affine'' is that 
each bisector in $\B$ has a finite midpoint,  whereas in \cite{OWPencil} we consider also bisector fields that can have a midpoint at infinity, and in such a case ``affine'' is omitted from the definition. 
It is not obvious from the definition that affine bisector fields exist, but a large source of examples is provided by one of the main theorems in \cite{OWPencil}: 
{\it A bisector field consists of the pairs of lines that occur as the degenerations of conics in a nontrivial pencil of conics} \cite[Theorem~6.3]{OWPencil}. For example, if ${\mathcal{C}}$ is  the set of conics through four points in general position, then the  asymptotes of the hyperbolas in the pencil are line pairs in a bisector field, and the only other line pairs in the bisector field are those that share an axis of symmetry with any degenerate parabolas that occur in the pencil.   
Therefore, a pencil of conics in the plane is asymptotically a bisector field, and every bisector field arises this way. 
All this works over any field of characteristic other than $2$.

The non-affine bisector fields 
 amount to a couple of pathological cases  that are less interesting than the affine case, and so to  streamline the present article we will exclude them and drop the adjective ``affine'' from here on out:

\medskip

{\noindent}{\bf  Standing assumption:} {\it Throughout the rest of the article, all bisectors fields are understood to be affine, as in Definition~\ref{bf}.}  

\medskip

A rather different source of bisector fields is   relevant for the proofs and theorems in the present paper.  By a bisector of a quadrilateral $Q = ABA'B'$, we mean a bisector of the set of pairs $\{\{A,A'\},\{B,B'\}\}$, that is, a line $\ell$ that is either a side of $Q$,   parallel to a pair of opposite sides of $Q$, or  crosses all four sides of $Q$ and satisfies  $\mid_{AA'}(\ell) = \mid_{BB'}(\ell)$. We denote the set of bisectors of $Q$ by $\Bis(Q)$. Theorem~\ref{ultimate} below asserts that every bisector field occurs as the set of bisectors of some quadrilateral. However, to state this theorem precisely,  it remains to indicate how to pair the bisectors of a quadrilateral $Q$. For this, we need the following  lemma. 

\begin{lemma} \label{sym}  {\em \cite[Section 2]{OWQuad}}
There is an inner product (a nondegenerate symmetric bilinear form) 
under which the slope vectors of 
%
opposite sides of $Q$ are orthogonal, as are the slope vectors of the  diagonals of $Q$. \end{lemma}

Two lines
are {\it $Q$-orthogonal} if their slope vectors are orthogonal under the inner product of the lemma. 
%
%
 The  notion of $Q$-orthogonality  is not quite sufficient to define a pairing on the bisectors of~$Q$. In order to complete the definition, we require a symmetry condition also: 
  A pair   of bisectors of $Q$ is 
    {\it $Q$-antipodal} if
the midpoint of the midpoints of the bisectors   is the centroid of $Q$. The {\it bisector locus}, the set of midpoints of  the bisectors of $Q$, is a conic (in fact, the nine point conic for $Q$) \cite[Theorem~5.2 and~Corollary 5.4]{OWQuad}, whose center is the centroid of $Q$. See Figure~1 for examples of bisector loci. 
 Thus two bisectors are $Q$-antipodal if and only if their midpoints lie on antipodal points of the bisector locus; see Figure~3.

 \begin{figure}[h] \label{diagonalspar}
 \begin{center}
 \includegraphics[width=0.7\textwidth,scale=.09]{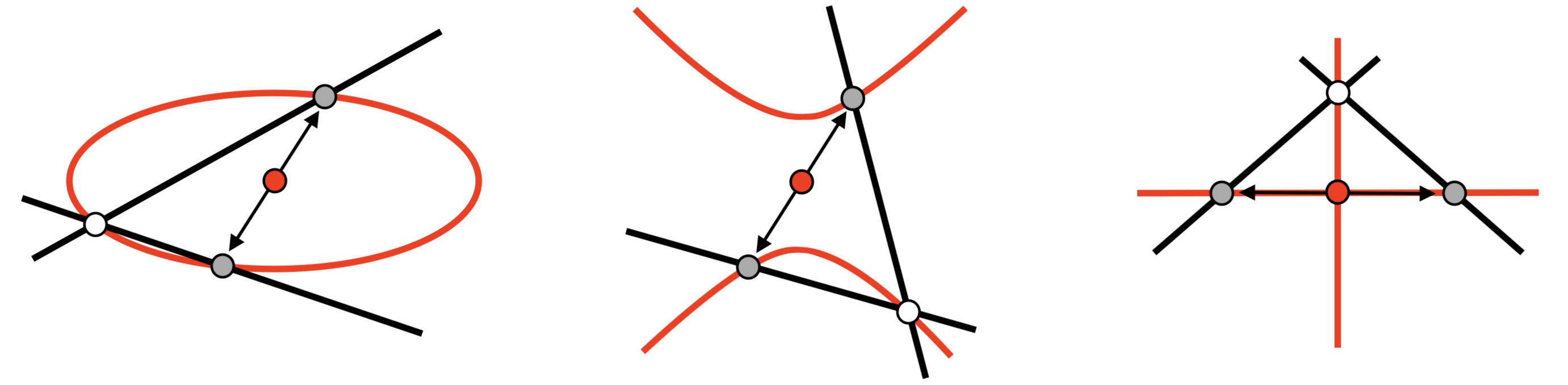} 
 \end{center}
 \caption{
$Q$-pairs of bisectors. (The quadrilateral $Q$ is not pictured.) The bisector locus for $Q$ is the red conic, and the red point is the center of the conic. $Q$-antipodal pairs of bisectors meet on the bisector locus.  
 }
\end{figure}

We can now define the pairing: A pair of bisectors  is a
 {\it $Q$-pair}  
if  it is  $Q$-antipodal and $Q$-orthogonal. 
In almost all cases,  a pair of bisectors is  $Q$-antipodal if and only if it is $Q$-orthogonal. This happens for example if no sides or diagonals of $Q$ are parallel \cite[Corollary~6.6]{OWQuad}.  If $\ell$ is a bisector of a quadrilateral $Q$, then there is a unique bisector $\ell'$ such that $\ell,\ell'$ is a $Q$-pair
\cite[Corollary~6.7]{OWQuad}. Each pair of opposite sides of $Q$ is a  $Q$-pair of bisectors, as is the pair of diagonals of $Q$; this follows from   \cite[Proposition~2.5]{OWQuad}.

Thus  we may view $\Bis(Q)$ as a collection of paired lines  whose pairing is given by $Q$-pairing.  
One of the main theorems in \cite{OWQuad} is that bisector pairs are themselves bisected, i.e., the $Q$-pairs of bisectors of $Q$ not only bisect $Q$ but are in turn bisected by every   bisector of $Q$. It is this symmetry that is the crucial property for a bisector field, and which is a part of the following theorem. 
By a {\it quadrilateral in $\B$} we mean a quadrilateral whose pairs of opposite sides are $\B$-pairs.

 \begin{theorem}  \label{ultimate}  
  \label{is a bisector field} {\em \cite[Theorem~6.9]{OWQuad}}
  If $\B$   is a bisector field (as in Definition~\ref{bf}), then  
  for each quadrilateral $Q$ in $\B$, $\B = \Bis(Q)$ and the pairing on $\B$ is  $Q$-pairing. 

\end{theorem}



  Any two quadrilaterals in the same bisector field have the same   centroid, and so we  designate this shared centroid as the {\it center} of the bisector field.   That all the quadrilaterals in a bisector field have the same centroid is a consequence of the fact that all these quadrilaterals share the same nine-point conic, and this conic has the centroid of the quadrilateral as its center; see~\cite[Theorem~5.2]{OWQuad}.

\section{The dual curve of bisectors}
\label{phi and psi}

For $n\geq 1$, we denote by $\P^n(\k)$ the projective closure of $\k^n$. The points in $\P^n(\k)$ are written in homogeneous coordinates $[x_1:\ldots:x_{n+1}]$, where $(x_1,\ldots,x_{n+1}) \ne (0,\ldots,0)$ and $[x_0:\ldots:x_{n+1}]$ is the line in $\k^{n+1}$
through the point $(x_1,\ldots,x_{n+1})$. 
Under point-line duality of the projective plane $\P^2(\k)$, each line $tX-uY+vZ =0$ in $\P^2(\k)$ corresponds to the point $[t:u:v]$ in the dual projective plane. Thus, in describing a bisector field, it it useful to consider the dual points of the bisectors. We do this in the present section and show that the coefficients of the  bisectors in a bisector field are encoded onto a singular cubic curve in $\P^2(\k)$. 

As discussed in the last section, to describe a bisector field, it is sufficient to describe the bisectors of any quadrilateral in the bisector field. 
We begin by   associating two polynomials to  a quadrilateral, the shape polynomial and the position polynomial. 
We will show in Corollary~\ref{switch} that these two polynomials completely determine the bisectors of $Q$. 
We use the following notational convention throughout the paper.

\smallskip

\begin{quote} 
{\it Given  
 a line  $L $ in the affine plane $\k^2$, we write an equation for $L$   as 
  $t_Lx-u_Ly+v_L=0,$ where $t_L,u_L,v_L \in \k$. In order to have a canonical choice of coefficients for the line, we assume that if $u_L = 0$ then $t_L=1$ and if $u_L \ne 0$ then $u_L=1$. The slope of $L$ will often be represented as the element $[t_L:u_L]$ in $ \P^1(\k)$.}
\end{quote}  

\smallskip
  
Let  $Q = ABA'{B'}$ be a quadrilateral, and consider the  polynomial in $\k[T,U]$ whose zeroes in $\P^1(\k)$ are all the slopes of the sides of $Q$:
  $$\Theta(T,U)=(u_{A}T-t_{A}U)(u_BT-t_BU)(u_{A'}T-t_{A'}U)(u_{B'}T-t_{B'}U).$$
 Let $\Theta_A$ be $\Theta$ with     the $A$ term $u_At-t_Au$ deleted, and let $\Theta_{A'}$ be $\Theta$ with the $A'$ term deleted. Define $\Theta_B$ and $\Theta_{B'}$ similarly, but by deleting the $B$ and $B'$ terms from $-\Theta$ rather than $\Theta$. 
 %
%

\begin{definition}  \label{shape def}
The {\it position polynomial}   for $Q$ is 
\begin{eqnarray*}
\Psi(T,U) =  \sum v_L \Theta_L(T,U), {\mbox{ where }} L {\mbox{ ranges over the sides of }} Q.    \end{eqnarray*} 
The {\it shape polynomial}  for $Q$ is the polynomial $\Phi(T,U) \in \k[T,U]$ for which  $$T\Phi(T,U)=\sum t_L\Theta_L(T,U) \:\: {\mbox{ and }} \:\: U\Phi(T,U) =   \sum u_L\Theta_L(T,U),$$  
Alternatively, $\Phi(T,U) = \alpha T^2-2\beta TU + \gamma U^2,$ where 
\begin{eqnarray*}
\alpha & = & t_Au_Bu_{A'}u_{B'} - u_At_Bu_{A'}u_{B'} + u_Au_Bt_{A'}u_{B'}  - u_Au_Bu_{A'}t_{B'} \\
%
\beta 
& = & 
 t_Au_Bt_{A'}u_{B'}-u_At_Bu_{A'}t_{B'}\\
\gamma 
& = &  t_At_Bt_{A'}u_{B'}-t_At_Bu_{A'}t_{B'}+t_Au_Bt_{A'}t_{B'}-u_At_Bt_{A'}t_{B'}.
  \end{eqnarray*}  
We write $\Phi_Q$ and $\Psi_Q$ for $\Phi$ and $\Psi$ when the quadrilateral is not clear from context. 
 \end{definition}
 
 That these two definitions for $\Phi$ are equivalent follows from the fact
  that 
$$\alpha = \sum t_L\Theta_L(1,0), \:\: \gamma = \sum t_L\Theta_L(0,1), \:\: -2\beta = \sum t_L\Theta_L(1,1)-\alpha-\gamma,$$ 
 and similarly for the polynomial $\sum_L u_L\Theta_L(T,U)$.  
 Although not needed in what follows, the polynomial $\Phi$ can also be viewed as the squared norm of the inner product
$\langle -,- \rangle_Q$  from Lemma~\ref{sym}.
This follows from the explicit description of the inner product given in \cite[Section~2]{OWQuad}.  


 The  shape polynomial $\Phi$ depends only on the slopes of the sides of $Q$, while the position polynomial $\Psi$ depends also on the position of $Q$. The reason for ``shape''  is that $\Phi$ gives the shape of the bisector locus for $Q$  which, if $Q$ is proper, coincides with the nine-point conic for $Q$.  
(The bisector locus is discussed in Section~2.)
Namely, 
 let $(h,k)$ be the centroid of $Q$ and let $(a,b)$ be a diagonal point of $Q$, that is, a point in $\k^2$ where a pair of opposite sides or diagonals of $Q$ meets.  It is shown in  \cite[Theorem 5.2]{OWQuad}
 that   the bisector locus of   $Q$   is a conic with center $(h,k)$  given by the equation $$\Phi(Y-k,X-h) = \Phi(b-k,a-h).$$     

 \begin{definition} For a quadrilateral $Q$,  the {\it dual polynomial} for $Q$ is the 
  homogeneous   polynomial  
 $\Psi(T,U)-V\Phi(T,U) \in \k[T,U,V].$ The {\it dual curve} of $Q$ is the zero set in $\P^2(\k)$ of the dual polynomial. 
\end{definition}

Since $\Psi$ and $\Phi$ are homogeneous in the variables $T$ and $U$ and the quadratic form $\Phi$ is never the zero polynomial (this follows from the nondegeneracy of the bilinear form in Lemma~\ref{sym} and the fact that $\Phi$ is the quadratic form associated to this inner product), the dual polynomial has degree~$3$ for any choice of quadrilateral~$Q$. Also, since $\Psi$ and $\Phi$ are homogeneous, the dual polynomial  has at least one singularity, namely the point $[0:0:1]$. The only other singularities   are those that appear when $\Psi$ and $\Phi$ share a common zero, in which case the cubic is not an irreducible curve. As the next theorem shows, the 
reason for the terminology of dual curve is that the points on this curve are the dual points for the bisectors of a quadrilateral.


\begin{theorem} \label{cubic}  
A line $tX-uY+v=0$ bisects a quadrilateral $Q$ if and only if $[t:u:v]$ is on the dual curve   of  $Q$.
%

 
 
 
\end{theorem}

\begin{proof} Write $Q=ABA'B'$.
 Suppose first the line  $\ell: tX-uY+v=0$  is not parallel to any of the sides of Q. For each line $L \in \{A,B,A',B'\}$,  define 
$$x_{L} = \frac{uv_L-u_Lv}{tu_L-t_Lu} \:\: \: { \rm{ and } } \:\:\: y_L = \frac{tv_L -t_L v}{tu_L-t_Lu}.$$
Since $\ell$ is not parallel to $L$, the line $\ell$ crosses   $L$ at $(x_L,y_L)$. 
The line  $\ell$ is a bisector of $Q$ if and only if the midpoint of the points where $\ell$ crosses $A$ and $A'$ is also the midpoint of the points where $\ell$ crosses $B$ and $B'$;
if and only if $x_A-x_B+ x_{A'} -x_{B'} =0$  and $y_A -y_B + y_{A'}  -y_{B'}   =0;$ if and only if 
$$ \sum (uv_L-vu_L)\Theta_L(t,u) =0 \:\: {\mbox { and } } \:\:  \sum (v_Lt-t_Lv)\Theta_L(t,u)=0;$$
if and only if 
$${\rm (a) } \: \: u\sum v_L\Theta_L (t,u)-v \sum u_L\Theta_L(t,u)=0 \:\: {\mbox { and } } \:\: {\rm (b) } \: \:
 t\sum v_L\Theta_L(t,u) - v\sum t_L\Theta_L(t,u)=0.$$
 By Definition~\ref{shape def}, $u\Phi(t,u) = \sum u_L\Theta_L(t,u)$ and $t\Phi(t,u) = \sum t_L\Theta_L(t,u)$.  Thus the equations in  (a) and (b)  become
  $$u\sum v_L \Theta_L(t,u)-vu\Phi(t,u) =0 \:\: {\mbox { and } } \:\: t\sum v_L \Theta_L(t,u) -vt\Phi(t,u)=0.$$
 Since $u$ and $t$ are not both $0$, the equations in (a) and (b) are valid (and hence $\ell$ is a bisector) if and only if $\sum v_L \Theta_L(t,u)-v\Phi(t,u) =0$. Since $\Psi(t,u) = \sum v_L \Theta_L(t,u)$, the theorem is proved under the assumption that $\ell$ is not parallel to any of the sides of $Q$.

Now suppose $\ell$ is parallel to a side $M$ of $Q$,  
 so that $\ell$ is defined by the equation $t_{M}x-u_My+v=0$ for some $v \in \k$. We claim $\ell$ is a bisector of $Q$ if and only if $[t_M:u_M:v]$ is on the dual curve of $Q$.
Using Definition~\ref{shape def} and the definition of the polynomials $\Theta_L$,   
\begin{eqnarray*}
u_M\Phi(t_M,u_M) &  = &   \sum u_L\Theta_L(t_M,u_M) \:\: = \:\:  u_M\Theta_M(t_M,u_M) \\
 t_M\Phi(t_M,u_M) & = & \sum t_L\Theta_L(t_M,u_M) \:\: = \:\:   t_M\Theta_M(t_M,u_M).
 \end{eqnarray*} 
 Since either $t_M\ne 0$ or $u_M \ne 0$, we conclude $\Phi(t_M,u_M) = \Theta_M(t_M,u_M).$

 Now $[t_M:u_M:v]$ is on the dual curve of $Q$ if and only if
  $v\Phi(t_M,u_M) = \Psi(t_M,u_M)$; if and only if $v \Theta_M(t_M,u_M) =  v_M\Theta_M(t_M,u_M)$; if and only if $v=v_M$ or $\Theta_M(t_M,u_M) =0$;  if and only if $\ell = M$ (since $\ell$ is parallel to $M$) or $M$ is parallel to another side of $Q$; if and only if $\ell$ is a bisector of $Q$. 
\end{proof}

As discussed in Section~2 (and proved in \cite[Corollary 6.10]{OWQuad}) two quadrilaterals that have the same bisectors have the same bisector field. The next corollary shows that the  dual polynomial can also be used to determine if two quadrilaterals have the same bisector field.

\begin{corollary}  \label{switch}  Let $Q$ and $Q'$ be quadrilaterals. Then $Q$ and $Q'$ have the same bisectors (and hence the same bisector fields) if and only if $Q$ and $Q'$ have the same dual polynomial;
if and only if  $\Phi_Q = \lambda \Phi_{Q'}$ and $\Psi_{Q}=\lambda \Phi_{Q'}$ for some 
 $\lambda \in \k$.  
\end{corollary} 

\begin{proof} 
For each $i=1,2$, let $\overline{Q_i}$   be the extension of $Q_i$  to the affine plane over the algebraic closure $\overline{\k}$ of $\k$.  Then 
$Q_i$ and $\overline{Q_i}$ share the same shape, position and dual polynomials since these are defined in terms of the coefficient data of the lines that comprise the sides of $Q_i$.
Since $\overline{\k}$ is infinite, Theorem~\ref{cubic} implies $\overline{Q_1}$ and $\overline{Q_2}$ share the same bisectors if and only if  the
dual curves of $\overline{Q_1}$ and $\overline{Q_2}$ are the same; if and only if the dual polynomial of $Q_1$ is a scalar multiple (over $\overline{\k}$) of the dual polynomial of $Q_2$; if and only if there is $\lambda \in \overline{\k}$ such that  $\Phi_{Q_1} = \lambda \Phi_{Q_2}$ and $\Psi_{Q_1}=\lambda \Phi_{Q_2}$. In this case, $\lambda \in \k$ since the coefficients of $\Phi_{Q_1}$ and $\Psi_{Q_2}$ are in $\k$. Thus
to prove the corollary it suffices to prove that $Q_1$ and $Q_2$ have the same bisectors if and only if $\overline{Q_1}$ and $\overline{Q_2}$ do. 

  Suppose $Q_1$ and $Q_2$ have the same bisectors. 
  For this part of the proof we rely on the discussion of pairing from Section~2. 
  By \cite[Corollary~6.10]{OWQuad}, $Q_1$-pairing and $Q_2$-pairing are   the same for these bisectors. Since 
  $\overline{Q_i}$-orthogonality is defined by the same inner product as $Q_i$-orthogonality (it is defined in terms of    the coefficients of the lines that are the sides of $Q_i$; see \cite[Section~2]{OWQuad}), and $Q_i$ and $\overline{Q_i}$ have the same centroid, it follows that the extensions $\overline{\ell}, \overline{\ell'}$ of two bisectors $\ell,\ell'$ of $Q_i$ form a $\overline{Q_i}$-pair if and only if $ \ell,\ell'$ form a $Q_i$-pair. Therefore,
  since the sides of $\overline{Q_1}$, being extended from bisectors of $Q_1$, and hence bisectors of  $Q_2$, are bisectors of $\overline{Q_2}$, it follows that 
   the quadrilateral $\overline{Q_1}$ is in the bisector field of $\overline{Q_2}$, which, as discussed in Section 2, implies $\overline{Q_1}$ and $\overline{Q_2}$ have the same bisectors. 
   Conversely, if $\overline{Q_1}$ and $\overline{Q_2}$ have the same bisectors, then, as we have already proved,  $Q_1$ and $Q_2$ have the same dual polynomials up to scalar multiple, and    hence $Q_1$ and $Q_2$ have the same bisectors by Theorem~\ref{cubic}.
 \end{proof}




\section{Cubic, quadratic and linear bisector fields}

Since a bisector field is the set of bisectors of a quadrilateral, 
Corollary~\ref{switch} allows us to shift our focus  from quadrilaterals to bisector fields and attach to bisector fields the polynomial data developed in the last section for quadrilaterals.  

\begin{definition} \label{bis not} For each  bisector field $\B$, fix a quadrilateral $Q$ such that $\B = \Bis(Q)$.  The {\it shape polynomial for $\B$}
is the shape polynomial $\Phi$ for $Q$, the {\it position polynomial for $\B$}   is the position polynomial $\Phi$ for $Q$,  the {\it dual polynomial for $\B$}
 is  the {dual polynomial} for~$Q$, and the {\it dual curve of $\B$} is the dual curve of $Q$.
\end{definition}

For the polynomials in Definition~\ref{bis not},  
 different quadrilaterals in $\B$ can yield different polynomials but these polynomials are by Corollary~\ref{switch}  scalar multiples of those for $Q$. Ultimately,  we are interested in zeroes of these polynomials (such as in Theorem~\ref{cubic})  so uniqueness up to scalar multiple is all that is needed. 
 
 The polynomials $\Psi$ and $\Phi$ may share common factors (see Lemma~\ref{zeroes}) and hence the dual polynomial from Theorem~\ref{cubic} may not be irreducible. We remove these factors in the formulation of the reduced dual polynomial, defined next. 
This irreducible polynomial is needed later to distinguish between moving bisectors and bisectors in a pencil of parallel lines. 
 As discussed in the
  last section, $\Phi$ is never the zero polynomial. However, it can happen that $\Psi = 0$ (see Lemma~\ref{orthogonal}), and this case matters in the next definition. 
  

      \begin{definition} \label{dual rational function def} Let $\B$ be a bisector field with shape and position polynomials $\Phi$ and $\Psi$.  
      If $\Psi \ne 0$, 
       let 
 $\phi$ and $\psi$ be homogeneous relatively prime polynomials in 
 ${\k}[T,U]$   such that $ \Phi \psi = \Psi \phi$ and $\phi$ is monic in $T$; otherwise, if $\Psi =0$, set $\psi(T,U) =0$ and $\phi(T,U)=1$. 
We define the {\it reduced dual polynomial} for $\B$ as
  $$F(T,U,V) = \psi(T,U)-V\phi(T,U).$$
The {\it reduced dual curve} of $\B$ is the zero set of $F$ in $\P^2(\k)$. 
  \end{definition}


The reduced dual polynomial $F$ is a factor of the dual polynomial. 
The degree of $F$, which is at most three, will be important in the next sections for describing the nature of the bisector field.


  \begin{definition} Let $\B$ be a bisector field, and let $F$ be the reduced dual polynomial for $\B$. Then $\B$ is {\it cubic}  if $\deg F =3$; {\it quadratic} if $\deg F=2$; and {\it linear} if $\deg F=1$. 
  \end{definition}
  
  A geometric interpretation of these classes of bisector fields is given in Theorem~\ref{new} and~Corollary~\ref{Bis cor}.

 We describe the properties of a special class of bisectors in the next lemma. 
Recall that a pencil of lines   is the set of lines through a point. This point can be at infinity, and so the set of all lines having a fixed slope is a pencil of parallel lines through a point at infinity.  


\begin{lemma}   \label{zeroes} Let $\B = \Bis(Q)$ be a bisector field, where $Q$ is a quadrilateral.  The following are equivalent for  a  bisector $\ell$ of $\B$. 
\begin{enumerate}
\item[$(1)$] $\ell$ is parallel to another bisector in $\B$. 
\item[$(2)$] $\ell$ is parallel to, and distinct from, a side or diagonal of $Q$. 
\item[$(3)$] $\ell$ is parallel to two sides or diagonals of $Q$. 
\item[$(4)$] $\ell$ belongs to a pencil of parallel bisectors in $\B$.
 
\item[$(5)$] $\Phi(t_\ell,u_\ell) =0$. 
\item[$(6)$] $\Phi(t_\ell,u_\ell) = \Psi(t_\ell,u_\ell) =0$.

\end{enumerate} 
\end{lemma} 

\begin{proof} The equivalence of (3), (4) and (5)  
 follows from \cite[Remark~2.6 and Lemma~5.3]{OWQuad}.
  That 
(2) implies (1) is clear since sides and diagonals of $Q$ are bisectors. To see that (1)  implies (5), assume the bisector $\ell$ is parallel to a different bisector $\ell'$. In this case,
$[t_\ell:u_\ell]=[t_{\ell'}:u_{\ell'}]$ and 
 $v_\ell \ne v_{\ell'}$, and so  Theorem~\ref{cubic} implies $\Phi(t_\ell,u_\ell) = 0$. That (3) implies (2) is clear, so   (1)--(5) are equivalent. 
%
To see that (5) implies (6), 
 use Theorem~\ref{cubic}. That (6) implies (5) is clear. \end{proof}

 \begin{definition}
A  bisector in $\B$ of a quadrilateral $Q$ is {\it null} if it satisfies the equivalent conditions of Lemma~\ref{zeroes}.
\end{definition}

Since the inner product in Lemma~\ref{sym} has $\Phi$ as its squared norm, the null bisectors of $Q$ are the bisectors that are $Q$-orthogonal to themselves.

\begin{theorem} \label{new} 
A bisector field is linear if and only if it contains exactly two pencils of parallel lines; it is quadratic if and only if it contains exactly one   pencil of parallel lines; and it is cubic if and only if it contains no   pencils of parallel lines. 
\end{theorem} 

\begin{proof} Let $\B$ be a bisector field, and let $F$ be the reduced dual polynomial of $\B$. By Theorem~\ref{cubic} and Lemma~\ref{zeroes}, if $[t:u] \in \P^1(\k)$, then there is a pencil of parallel lines in $\B$ of slope $[t:u]$   if and only if $\Phi(t,u)=$ $\Psi(t,u) = 0$. Thus to 
 prove the theorem, it suffices   to show that
the number of mutual zeroes of $\Phi$ and $\Psi$ in $\P^1(\k)$  is  $3-\deg F$.
By the choice of $\phi$, $2-\deg \phi$ is the number of mutual zeroes of $\Phi$ and $\Psi$ in $\P^1(\k)$.  
Thus, since
$\deg F = 1 + \deg \phi$, we have that $3-\deg F = 2-\deg \phi$ is the number of mutual zeroes of $\Phi$ and $\Psi$, which proves the theorem.  
\end{proof}

Theorem~\ref{new} implies that
 in Figure~1, the bisector fields in (a) and (b) are cubic bisector fields; the bisector field in (c) is quadratic; and the bisector field in (d) is linear. 


\begin{corollary} \label{Bis cor} Let $\B$ be a bisector field. 
\begin{enumerate} 
\item[$(1)$]  $\B$ is cubic if and only if
$\B$ is the bisector field of a quadrilateral  with no parallel sides or diagonals. 
\item[$(2)$] 
$\B$ is quadratic if and only if  $\B$ is the bisector field of a  trapezoid   that is not a parallelogram.
\item[$(3)$] $\B$ is linear if and only if $\B$ is the bisector field of a   parallelogram.
\end{enumerate}
\end{corollary}

\begin{proof} By Lemma~\ref{zeroes} and Theorem~\ref{new},  
 $\B$ is cubic if and only if  no bisector is null, and so (1) follows from Lemma~\ref{zeroes}.   
For (2), suppose $\B$ is quadratic. By Theorem~\ref{new}, there is  a $\B$-pair of distinct parallel bisectors $A$ and $A'$ in $\B$.   Let $B,B'$ be a $\B$-pair of non-null bisectors. Since $B$ and $B'$ are not parallel to each other or any other bisectors, $Q = ABA'B'$ is a trapezoid that is not a parallelogram, and by Theorem~\ref{ultimate}, $\B = \Bis(Q)$. For the converse,   Lemma~\ref{zeroes}  implies that $\B$ contains exactly one  pencil of parallel lines and so $\B$ is quadratic by 
Theorem~\ref{new}.  
The proof of statement (3) is similar to that of (2).  
\end{proof}

By Lemma~\ref{zeroes}, each null bisector in a bisector field is in a pencil of parallel lines in that bisector field. Thus the set of null bisectors is either empty, a single pencil of parallel lines or a union of two pencils of parallel lines. 
Non-null bisectors have a more interesting description, and these will described   in Sections~6 and~7 with the notion of moving bisectors.

\section{Affine equivalence }

 Since affine transformations preserve midpoints and parallel lines, the image of a bisector field under such a transformation   is again a bisector field.  
Two bisector fields  $
\B_1$ and $\B_2$ are {\it affinely equivalent} if there is an affine transformation of the plane that carries the lines in  $\B_1$ onto the lines in $\B_2$. The pairing on the lines in $\B_1$ becomes a pairing on the lines in $\B_2$, and as discussed in Section~2 this is the only possible pairing on $\B_2$ for which $\B_2$ is a bisector field. We will use  affine transformations to reduce to bisector fields of a form in which the shape and position polynomials $\Phi$ and $\Psi$ become more tractable. 
This is helpful because the coefficients of  a typical polynomial in our context  needs to be treated as variables in the course of proofs and so these equations can involve large expressions. 
For example, in the general case, with the coefficients taken as variables, the position polynomial has fifteen variables, degree seven and too many terms to list reasonably. Lemma~\ref{orthogonal} shows that finding the right transformation for the bisector field greatly simplifies these expressions.




\begin{definition} \label{normal} A bisector field $\B$ is in {\it standard form} if the axes $X=0$ and $Y=0$ are a $\B$-pair of bisectors   in $\B$. 
\end{definition} 


 Every bisector field $\B$ is affinely equivalent to a bisector field in standard form. 
%
This is because there is a $\B$-pair of    lines $\ell$ and $\ell'$ in $\B$ that are not parallel, and so after a translation of the plane, we may assume $\ell$ and $\ell'$ meet at the origin. An invertible linear transformation carries $\ell$ onto the line $Y=0$ and $\ell'$ onto  $X=0$, and so the image of $\B$ under this transformation is a bisector field in standard form. We will use the following lemma often in what follows.


 \begin{lemma} \label{orthogonal} Let $\B$ be a bisector field in standard form with center $(h,k)$. 
 Then 
$t_\ell t_{\ell'}$ has the same value $\mu$ for all $\B$-pairs  of bisectors $\ell, \ell'$  distinct from the $X$ and $Y$-axes, and 
$$\Phi(T,U) = T^2-\mu U^2 \: {\mbox{ and }} \: \Psi(T,U) = 4TU(
kT+ \mu h U).$$

\end{lemma}

 \begin{proof} 
Let $A$ be the line $Y=0$ and $A'$  the line $X=0$, and let $B,B'$ be a $\B$-pair of distinct bisectors in $\B$, neither of which is parallel to $A$ or $A'$.     
Then $t_A=u_{A'}=v_A=v_{A'}=0$ and $u_A=t_{A'}=u_B=u_{B'}=1$. 
Let $\mu =t_Bt_{B'}$.  
The vertices of the quadrilateral $Q = ABA'B'$ are 
 \begin{eqnarray} \label{vertices}
 A \cdot B = \left(-\frac{t_{B'}v_B}{\mu},0\right), \: B \cdot {A'} = (0,v_B), \: {A'} \cdot {B'} = (0,v_{B'}), \: {B'} \cdot A = \left(-\frac{t_Bv_{B'}}{\mu},0\right).
 \end{eqnarray}  Thus the centroid $(h,k)$ of $Q$, which is the center of $\B$, is given by the equations $-4h\mu = t_{B'}v_B + t_Bv_{B'}$ and $4k = v_B+v_{B'}$. 
 Substituting all this data into the equations for $\Psi$ and $\Phi$ from Section~\ref{phi and psi} yields $\Phi(T,U) = T^2-\mu U^2$ and $ \Psi(T,U) = 4TU(
kT+ \mu h U).$ In particular,  a $\B$-pair $\ell,\ell'$ of bisectors in $\B$, since these bisectors are $Q$-orthogonal and $\Phi(T,U)=T^2-\mu U^2$,  satisfies $t_\ell t_{\ell'} -\mu u_\ell u_{\ell'}=0$ (see \cite[Definition~2.3]{OWQuad}).

. If $\ell$ and $\ell'$ are not parallel to $A$ or $A'$, we have $u_\ell =u_{\ell'} = 1$ and hence $\mu=t_\ell t_{\ell'}$, and so the value $\mu$ is independent of the choice of $\ell$ and $\ell'$.  
 \end{proof}

The quantity $\mu$ in Lemma~\ref{orthogonal} is independent of the choice of $\B$-pair, as long as the pair is not the pair of axes. Because of this, we can make the following definition. 

\begin{definition}  \label{coeff def} The {\it coefficient}  of  a bisector field $\B$ in standard form is the product 
 $\mu = t_\ell t_{\ell'}$ of the slopes $t_\ell$ and $t_\ell'$ of any $\B$-pair $\ell,\ell'$  of bisectors in $\B$ not parallel to $X=0$ or $Y=0$. 
\end{definition}

 The coefficient of the bisector field in Figure~1(a) is  $-1$ since all the $\B$-pairs in this bisector field are perpendicular. The coefficient of the bisector field in Figure~1(b) is $1$ since the lines in the $\B$-pairs are reflected about the line $Y=X$.  The next theorem is important for the classifications of bisector field in Section~8.

 \begin{theorem} \label{equal} Two bisector fields in standard form are the same if and only if they have the same   center and the same coefficient.
 \end{theorem}
 
 \begin{proof} By Corollary~\ref{switch}, two bisector fields are the same if and only if they have the same shape and position polynomials up to scalar multiple. This fact 
  and Lemma~\ref{orthogonal} imply the theorem. 
 \end{proof}

 The theorem shows the center and coefficient uniquely determine a bisector field in standard form, and so when 
seeking to determine whether two bisector fields are affinely equivalent it is natural  to do so  using this data, as in the next lemma, which will be important in Section~8 as well as in the rest of this section.
 





      \begin{lemma}  \label{preserve}   Let $\B$ be a bisector field in standard form with coefficient $\mu_1$, and let $0\ne \mu_2 \in \k$ such that $\mu_1\mu_2 $ is a square in $\k$. 
If $\ell,\ell'$ is a $\B$-pair of bisectors that are not parallel, then there is an affine transformation of the plane  that sends the lines
   $\ell$ and $\ell'$   to the lines  $X=0$ and $Y=0$ and sends $\B$ to a bisector field in standard form with coefficient $\mu_2$.   
\end{lemma}

\begin{proof} Let ${\mathbb{B}}_1 = \B$, and let $\mu_1$ be the coefficient of $\B_1$.  
Write the slope of $\ell$ as $[t:u]$ and that of $\ell'$ as $[t':u']$. Then $tt' - \mu_1uu'=0$ since $\ell,\ell'$ is a $\B_1$-pair and $\mu_1$ is the coefficient of $\B_1$. Also, since $\ell$ and $\ell'$ are not parallel, $\ell$ is not null (Lemma~\ref{zeroes}), and so $\Phi(t,u)=t^2- \mu_1 u^2 \ne 0$ by Lemma~\ref{zeroes}. 
By assumption, there is $\theta \in \k$ such that $1=\theta^2  \mu_1 \mu_2$.  
Let ${\mathcal{L}}$ be the linear transformation of the plane given by  \begin{center} ${\mathcal{L}}(x,y) = (\mu_1\theta(tx-uy),-\mu_1ux+ty)$
for all $x,y \in k$.  \end{center}
  Since the determinant of this linear transformation is 
  $\mu_1\theta(t^2-\mu_1u^2)$ and no factor in this product is $0$, 
  the transformation ${\mathcal{L}}$ is invertible. Thus the  image    of $\B_1$ under ${\mathcal{L}}$ is a bisector field whose pairing is induced by that of $\B_1$.  
  
 To determine the slopes of the lines ${\mathcal{L}}(\ell)$ and ${\mathcal{L}}(\ell')$, observe that since   $tt' -\mu_1 uu' = 0$,  
  \begin{center}
  ${\mathcal{L}}(u,t) = (0,t^2-\mu_1 u^2)$  and   ${\mathcal{L}}(u',t') = (\mu_1\theta(  tu'-ut'),0).$
  \end{center}
Thus  ${\mathcal{L}}(\ell)$ is parallel to $X=0$ and ${\mathcal{L}}(\ell')$ is parallel to $Y=0$.  Let ${\mathcal{A}}$ be the affine transformation that is the composition of ${\mathcal{L}}$ with the translation that carries the intersection of ${\mathcal{L}}(\ell)$ and ${\mathcal{L}}(\ell')$ to the origin and hence carries these two lines to the lines $X=0$ and $Y=0$. Let $\B_2$ be the image of $\B_1$ under ${\mathcal{A}}$.  
Since $\ell,\ell'$ is a $\B_1$-pair,  the lines $X=0$ and $Y=0$ form a $\B_2$-pair of bisectors in $\B_2$. 
Thus  $\B_2$   is a bisector field in standard form.

It remains to  show $\B_2$   has coefficient $\mu_2$.  
 Let $\ell_1$ and $\ell_2$ be a $\B_1$-pair of  bisectors in~$\B_1$, and let $[t_1:u_1]$ and $[t_2:u_2]$ be the slopes of these two lines.
 For each $i=1,2$, let $\upsilon_i = \mu_1 \theta (tu_i-ut_i)$ and 
 $\tau_i = -uu_i \mu_1 +tt_i$. 
Then   
${\mathcal{L}}(u_i,t_i) = (\upsilon_i,\tau_i),$ and so ${\mathcal{L}}(\ell_i)$ has slope $[\tau_i:\upsilon_i]$.  
Now $t_1t_2-\mu_1u_1u_2=0$ since $\ell_1,\ell_2$ is a $\B_1$-pair of bisectors in $\B_1$. Using this fact, along with  
 fact that $\mu_1\mu_2\theta^2=1$, a calculation shows
$$\tau_1\tau_2-\mu_2 \upsilon_1\upsilon_2  = 
(\mu_1u^2-t^2)(\mu_1u_1u_2-t_1t_2)=0.$$
Since this is the case for all $\B_1$-pairs $\ell_1,\ell_2$, it follows that the coefficient of $\B_2$ is $\mu_2$.  
%
%
\end{proof}

The following theorem gives a simple  criterion for when a bisector field in standard form  can be transformed into one with a specified coefficient. 

         \begin{theorem} \label{big question lemma}   Let $\B_1$ be a bisector field in standard form with coefficient $\mu_1$, and let $0 \ne \mu_2 \in \k$. Then
         $\mu_1\mu_2$ is a square in $\k$ if and only if 
        there is a bisector field in standard form with coefficient $\mu_2$ that is    affinely equivalent to $\B_1$. 
     
%
         \end{theorem}
                
     
         \begin{proof} 
     Suppose $g:\k^2 \rightarrow \k^2:(x,y)\mapsto (ax+by+e,cx+dy+f)$ is an affine transformation that sends $\B_1$ to  a bisector field $\B_2$ in standard form with coefficient $\mu_2$. 
     Then $g$ sends the line $Y=0$ to the line whose slope is $[c:a]$ and the line $X=0$ to the line whose slope is $[d:b]$. 
     Since the $X$ and $Y$-axes are a $\B_1$-pair of bisectors in $\B_1$, the images of these two lines form a $\B_2$-pair of bisectors in $\B_2$.  
     Thus $cd-\mu_2 ab =0$.  The line that maps onto the line $X=0$ is  
     $aX+bY+e=0$ and 
     has slope   $[a:-b]$. Similarly,  the line that maps onto $Y=0$ is $cX+dY+f=0$ and has slope $[c:-d]$.  These lines form a $\B_1$-pair of bisectors, so $ac-\mu_1bd=0$.   With the aim of showing $\mu_1\mu_2$ is a square in $\k$, we consider solutions $a,b,c,d$ to the equations $cd-\mu_2 ab =0$ and $ac-\mu_1bd=0$. We will use that fact that $ad -bc \ne 0$ since $g$ is invertible. 
     
 First suppose   $b = c =0$. In this case, $a \ne 0$ and $d \ne 0$ and 
 $g(x,y) = (ax+e,dy+f)$ for all $(x,y) \in \k^2$. 
 Let $B,B'$ be a $\B_1$-pair of bisectors in $\B_1$ that are not parallel to $X=0$ or $Y=0$. Then $u_B = u_{B'} =1$ and the image of $B$ under $g$ has slope $dt_B/a$ while the image of $B'$ has slope $dt_{B'}/a$. 
 This implies $d^2t_{B}t_{B'} = a^2\mu_2$ since the images of $B$ and $B'$ form a $\B_2$-pair. Thus  
   $d^2\mu_1 -a^2 \mu_2 = d^2t_Bt_{B'}-a^2\mu_2 =0$.    Since $d \ne 0$, this implies $\mu_1/\mu_2$ is a square in $\k$, and so  $\mu_1\mu_2$ is a square in $\k$ since $\mu_2 \in \k$.
A similar argument shows that  if $a =d =0$, then $\mu_1\mu_2$ is a square in $\k$. 

   Finally, suppose at least one of $b$ and $c$ is nonzero and at least one of $a$ and $d$ is nonzero. In this case, the equations $cd-\mu_2 ba =0$ and $ac-\mu_1bd=0$ imply
   there is an element $\theta$ of the algebraic closure of $\k$ such that $\mu_1\mu_2\theta^2 = 1$,  $a =\mu_1 \theta d$ and  $b =\theta c$. If $d =0$, then $c \ne 0$ since $ad -bc \ne 0$. In this case, since $\theta c =b \in \k$, we have $\theta \in \k$ and hence $\mu_1 \mu_2 = \theta^{-1} \in \k$. Similarly, if $c =0$, then $d \ne 0$ and $\mu_1 \mu_2=\theta^{-1} \in \k$.  The converse of the theorem follows from Lemma~\ref{preserve}.
   \end{proof}

     One motivation for the next definition is that the existence of a bisector through the center of  the bisector field allows for a significant simplification in calculations. This is because, as we will see later,  it is possible to take such a bisector field and transform it into one whose center is on the $Y$-axis, thus introducing an extra zero into the calculations.

\begin{definition} A bisector field    is {\it well centered}
    if there is a bisector   passing through its center.
    \end{definition}
    
    A linear or quadratic bisector field is always well centered. This is because in both of these cases, the bisector field contains a pencil of parallel lines and hence every point in the plane has a bisector through it. Thus the question of whether a bisector field is well centered is of consequence only for cubic bisector fields.  See Example~\ref{new ex} for an example of a bisector field that is not well centered.

    Since the center of a bisector field remains a center under affine transformation, whether a cubic bisector field is well centered is invariant under such transformations, and so we can reduce to the case of a cubic bisector field in standard form. In this case there is a straightforward criterion for when the bisector field is well centered.

    \begin{lemma} \label{well lemma} A  cubic bisector field in standard form with center $(h,k)$ and coefficient $\mu$  is well centered if and only if there is a zero in $\P^1(\k)$ of the polynomial
      \begin{eqnarray} \label{well} 
    hT^3 + 3kT^2U + 3h\mu T U^2
    +k \mu U^3.
    \end{eqnarray}
Moreover, every bisector field over the field $\k$ is well centered 
 if and only if 
the polynomial $T^3+3T^2 +3\mu T+\mu$ has a zero in $\k$ for every $\mu \in \k$.  
    \end{lemma}
    
    \begin{proof} 
By Theorem~\ref{cubic} and Lemma~\ref{orthogonal}, 
 a cubic bisector field  $\B$ is well centered if and only if there is a zero in $\P^1(\k)$ of the polynomial $$hT(T^2-\mu U^2) -kU(T^2-\mu U^2) + 4TU(kT+\mu hU).$$ This polynomial is the same as that in $(\ref{well})$. 
To prove the second assertion, let $\B$ be a bisector field. We can  assume $\B$ is a cubic bisector field in standard form with center  $(h,k)$, since linear and quadratic bisector fields are always well centered.   If $h =0$ or $k=0$, then $\B$ is well centered since the axes are in $\B$. Otherwise, if $h \ne 0$ and $k \ne 0$, then after a linear transformation that scales each coordinate we can assume $(h,k) = (1,1)$. (A bisector field in standard form remains in standard form when scaled.) By the first claim, $\B$ is well centered if and only if $T^3+3T^2 +3\mu T+\mu$ has a zero.
\end{proof}  

   Thus over an algebraically closed field every cubic bisector field is well centered. This holds true  over a real closed field  $\k$ also since every cubic over $\k$ has a root in $\k$. 
As another example, if $\k$ has characteristic $3$, then every bisector field is well centered if and only if $\k$ is closed under cube roots. 

\begin{example} \label{new ex}
    Not all bisector fields are well-centered. Let $\k$ be the field with $7$ elements, and let $\B$ be the unique (cubic) bisector field in standard form with $h = k = 1$ and  $\mu =2$. The polynomial in Lemma~\ref{well lemma} below
 has no zero in $\k$, and so $\B$ is not well centered.  Similarly, if $\k$ is the field of rational numbers and $\B$ is the unique  bisector field  in standard form with  $h=k=\mu=1$, then the polynomial in Lemma~\ref{well lemma} has one real root in $\P^1({\mathbb{R}})$ but no roots in $\P^1({\mathbb{Q}})$, and so $\B$ is not well centered. 
 \end{example}

 The next theorem shows that for well centered bisector fields, the geometric problem of classification of bisector fields up to affine equivalence is also an arithmetical one.  The difference between this theorem and Theorem~\ref{big question lemma} is that in the latter we could only assert equivalence with {\it some} bisector field having coefficient $\mu_2$, whereas here we get equivalence with {\it every} such bisector field.

          \begin{theorem} \label{big question} Let $\B_1$ and $\B_2$ be  well-centered cubic  bisector fields in standard form
     with coefficients $\mu_1$ and~$\mu_2$, respectively.
     Then  $\B_1$ is affinely equivalent to $\B_2$ if and only if $\mu_1\mu_2$ is a square   in $\k$. 
     \end{theorem}
     
         \begin{proof}  If $\B_1$ is affinely equivalent to $\B_2$, then $\mu_1\mu_2$ is a square by Theorem~\ref{big question lemma}. 
     Conversely, suppose  
 $\mu_1\mu_2$ is a square in $\k$. 
Since $\B_1$ is well-centered, there is a line $A$ through the center   of $\B_1$.  Let $A'$ be the bisector in $\B_1$ such that $A,A'$ is a $\B_1$-pair. Since $\B_1$ is a cubic bisector field, $\B_1$ has no null bisectors, and so $A$ and $A'$ are not parallel. 
 By Lemma~\ref{preserve}, $\B_1$ is  affinely equivalent to a bisector field in standard form having the same coefficient as $\B_1$ and whose center $(h,k)$ lies on the $Y$-axis. Thus $h=0$, and since~$\B_1$ is a cubic bisector field, Lemma~\ref{orthogonal}
   implies 
  the center of $\B_1$ is not the origin since otherwise the reduced dual curve of $\B_1$ will have degree $1$, contrary to  assumption. Thus $k \ne 0$.  
  Since uniformly scaling the plane does not change the coefficient of a bisector field, $\B_1$ is affinely equivalent to the  cubic bisector field  in standard form with center $(0,1/2)$ and coefficient $\mu_1$.  
 Applying this same reasoning to $\B_2$, we 
 can assume without loss of generality 
   that both $\B_1$ and $\B_2$   have center   $(0,1/2)$.

 The fact that $\mu_1\mu_2$ is a square in $\k$ implies  
    there is $\theta \in \k$ such that $\mu_2=\theta^2 \mu_1$.   
Define an affine transformation by     
     $g:\k^2 \rightarrow \k^2:(x,y)\mapsto (\theta^{-1} x,y).$ 
Then $g(0,1/2) = (0,1/2)$.  Let $B,B'$ be a $\B_1$-pair of bisectors in $\B_1$ that are not parallel to $X=0$ or $Y=0$.  Then $g(B),g(B')$ 
is a $g(\B_1)$-pair of bisectors in the bisector field $g(\B_1)$. The definition of the map $g$ shows neither line in this pair is 
 parallel  to $X=0$ or $Y=0$, so since $B,B'$ is a $\B_1$-pair and neither $B$ nor $B'$ is parallel to the $X$ or $Y$-axes, we have $\mu_1 = t_B t_{B'}$. 
Now $g(B)$ has slope $\theta t_B$ and $g(B')$ has slope $\theta t_{B'}$. Thus the coefficient of the bisector field $g(\B_1)$ that is the image of $\B_1$ is $ t_{g(B)} t_{g(B')} = \theta^{2} t_Bt_{B'} =\theta^{2} \mu_1=  \mu_2,$ and so the image of $\B_1$ has center $(0,1/2)$ and coefficient $\mu_2$. Since this image is in standard form (as $g$ sends the $X$ and $Y$-axes to themselves), we conclude from Theorem~\ref{equal}
   that the image of $\B_1$ is $\B_2$.  
     \end{proof}

\section{The boundary of a nonlinear bisector field} \label{envelope}


 
 

Let $\B$ be a bisector field, and let $\phi$ and $\psi$ be as in Defintion~\ref{dual rational function def}. Theorem~\ref{cubic} implies
the lines of the form 
\begin{eqnarray} \label{moving eq}
t\phi(t,u)X-u\phi(t,u)Y+\psi(t,u), \:\: {\mbox{ where }} [t:u] \in \P^1(\k) {\mbox{ and }} \phi(t,u) \ne 0,
\end{eqnarray}
are bisectors in $\B$. These lines  form a ``moving line'' in the sense of \cite{CSC}.    The terminology  is motivated by the fact that the map $[t:u] \mapsto t\phi(t,u)X-u\phi(t,u)X+\psi(t,u)$ from $\P^1(\k)$ to the set of lines in $\k^2$    produces a flow of   bisectors.   

\begin{definition} \label{moving} The {\it moving bisectors} in a bisector field   are the bisectors in $(\ref{moving eq})$. 
\end{definition}

The dual points of the moving bisectors are the points on the reduced dual curve: 

\begin{proposition} \label{new moving bisector} 
A line $\ell:tX-uY+v$ is a moving bisector in a bisector field $\B$ if and only if $[t:u:v]$ is a zero of the reduced dual polynomial $F(T,U,V)=\psi(T,U)-V\phi(T,U)$ for $\B$.  
\end{proposition}  

\begin{proof} The line $\ell$ is a moving bisector if and only if $v \phi(t,u) = \psi(t,u)$ and $\phi(t,u) \ne 0$. The condition $\phi(t,u) \ne 0$ is redundant since by the choice of $\phi$ and $\psi$, these polynomials cannot share a  zero in $\P^1(\k)$.  Thus $\ell$ is a moving bisector if and only if $F(t,u,v)=0$.
\end{proof}

By Theorem~\ref{new}, a bisector field is cubic if and only if every bisector is a moving bisector.  For each of the bisector fields in Figure 1, the moving bisectors are the lines that are not parallel to any other lines, along with the midlines of the pairs of parallel lines.

We define a ``boundary'' of a bisector field in terms of the (homogeneous) discriminant of its moving line of bisectors.
As the presence  of the discriminant suggests, the boundary will be the envelope of the moving bisectors. We work out this connection with the envelope in the next section.

\begin{definition} \label{Delta notation}
The {\it boundary} of a nonlinear bisector field $\B$ is the curve defined by  the homogeneous discriminant $\Delta(X,Y)$  of the moving bisectors of $\B$, i.e., the homogeneous discriminant of    $T\phi(T,U)X-U\phi(T,U)X+\psi(T,U)$
  viewed as a polynomial in variables $T$ and $U$ with coefficients in $\k[X,Y]$. If $\B$ is linear,  the {\it boundary} of $\B$ is the center of~$\B$. 
\end{definition}

 \begin{theorem} \label{parabola theorem}  
 The boundary of a quadratic bisector field is a non-degenerate parabola. There is an affine transformation of the plane that carries the boundary onto the parabola $Y=X^2$ and the pencil of null bisectors onto the pencil of lines parallel to the $X$-axis.  
\end{theorem}

\begin{proof}  Let $\B$ be a quadratic bisector field. It suffices to prove the claim in the second sentence, and for this we may assume   $\B$ is in standard form. 
We show first that 
 $h \ne 0$, $k \ne 0$, $\mu = k^2/h^2$, the slope of  every null bisector   is $-k/h$,   
  $\phi(T,U) = T- kh^{-1} U$ and $\psi(T,U) = 4kTU.$
Let $A$ be the line $Y=0$ and $A'$  the line $X=0$. By Lemma~\ref{orthogonal},  $\Phi(T,U) = T^2-\mu U^2$ and $ \Psi(T,U) = 4TU(
kT+ \mu h U).$  
 Now $(h,k) \ne (0,0)$ since otherwise $\Psi(T,U) =0$, contrary to the fact that $\B$ is not linear. Let
$B,B'$ be a
$\B$-pair of distinct bisectors in $\B$ such that $B$ is null.
It follows from Lemma~\ref{zeroes} that $B$ is $Q$-orthogonal to itself, where $Q$ is a quadrilateral such that $\B$ is the bisector field of $Q$. Since $B$ is $Q$-orthogonal to $B'$ also, $B$ is parallel to $B'$. By Lemma~\ref{zeroes},  neither $B$ nor $B'$  is parallel to $A$ or $A'$ since $A$ is not parallel to $A'$.   
Since $t_B=t_{B'}$, the   calculation of the centroid of the quadrilateral $Q = ABA'B'$  in the proof of Lemma~\ref{orthogonal} shows    $-4 h \mu = t_B(v_B+v_{B'}) = 4t_Bk$, so that $h\mu = -t_Bk$, and hence $t_Bh = -k$ since  $\mu = t_Bt_{B'} = t_B^2$. 
Since $(h,k) \ne (0,0)$, this implies neither $h$ nor $k$ is $0$. Also, the fact that $t_B h = -k$ implies the slope of $B$ is $-k/h$ and $\mu = t_B^2 = k^2/h^2$.  
Moreover,  $$\Psi(T,U) = 4TU(kT+\mu hU) = 
 4kTU(T-t_BU).$$ 
Since also  $$\Phi(T,U)=T^2-\mu U^2 = (T-t_BU)(T+t_BU),$$
this implies $\psi(T,U) = 4kTU$ and $\phi(T,U) = T+t_BU = T-kh^{-1}U$.

The moving line    from Definition~\ref{Delta notation} is
 \begin{eqnarray*} 
  (TX-UY)(T-kh^{-1}U)+4kTU 
   =  XT^2+(4k- kh^{-1}X-Y)TU+kh^{-1}YU^2.
  \end{eqnarray*}
Calculating the homogeneous discriminant of this polynomial viewed as a quadratic  polynomial with coefficients in $\k[X,Y]$, we find
the boundary of $\B$ to be the curve defined by    
\begin{eqnarray*}
\Delta(X,Y) & = &  
\left(kh^{-1}X-Y\right)^2
-8k(kh^{-1}X+Y-2k).
 \end{eqnarray*}
Let $f(X,Y) = kh^{-1}X-Y$ and $g(X,Y) =-8k(kh^{-1}X+Y-2k).$ 
 Then  $\Delta(X,Y) = f(X,Y)^2-g(X,Y),$ so since $f$ and $g$ are linear and $-16k^2h^{-1} \ne 0$, the  transformation
   ${\mathcal{A}}$ given by ${\mathcal{A}}(x,y) = (f(x,y),g(x,y))$ is affine and
    carries the curve $\Delta$ onto the non-degenerate parabola $Y=X^2$.      
Since  $\Delta_X(h,k) = -{8k^2}{h^{-1}}$ and $\Delta_Y(h,k) = -8k$, the line $\ell$ tangent to $\Delta$ at $(h,k)$ has slope $[-k:h]$, which, as established earlier in the proof, is the slope of the null bisectors. The image of $(h,k)$ under ${\mathcal{A}}$ is $(0,0)$ and the image of $\ell$ is the $X$-axis, so the theorem follows. 
\end{proof}

The boundary of a cubic bisector field is more complicated:

\begin{lemma} \label{tedious} 
Let  $\B$ be a cubic bisector field in standard form. Let $(h,k)$ denote the center of $\B$, and let $\mu$ be the coefficient of $\B$. 
The boundary of $\B$ is   
\begin{eqnarray*}
\Delta(X,Y) & = &
4\mu(\mu^2 X^4- 12h\mu^2 X^3-2\mu X^2Y^2
-20k\mu X^2Y  +4\mu(12h^2\mu+k^2)X^2
\\
 & &
-20h\mu XY^2+
88hk\mu XY-
32h\mu(2h^2\mu+k^2)X+
Y^4-12kY^3 \\
 & &
 +4(h^2\mu+12k^2)Y^2-32k(h^2\mu+2k^2)Y+64h^2k^2\mu).
\end{eqnarray*}
%
 \end{lemma} 
 
 \begin{proof} Since $\B$ is in standard form,  Lemma~\ref{orthogonal} implies the moving line  in Definition~\ref{Delta notation} is 
 \begin{eqnarray*} 
  XT^3+(4k-Y)T^2U+\mu(4h-X)TU^2+\mu Y U^3.
 \end{eqnarray*}
%
%
A calculation shows the homogeneous  discriminant $\Delta(X,Y)$ of this cubic, viewed as a polynomial in $T$ and $U$, is as in the lemma. 
 \end{proof}

 As illustrated by Figures 1(a) and 1(b), the boundary of a cubic bisector field can have singularities, unlike the boundary of a quadratic bisector field. 

 \begin{theorem} \label{Delta theorem} Suppose $\k$ is algebraically closed. If $\B$ is a cubic bisector field, then $\B$ is affinely equivalent to a bisector field whose boundary   is the rational quartic  curve \begin{eqnarray*}  
 X^4+2X^2Y^2+Y^4+10X^2Y-6Y^3-X^2+12Y^2-8Y = 0.
\end{eqnarray*}
If {\rm char}~$\k \ne 3$,  the boundary has three singularities, all of which are ordinary cusps. If {\rm char}~$\k = 3$,   there is only one  singularity, and this singularity is a higher-order cusp.
\end{theorem}

\begin{proof}
Since $\k$ is algebraically closed, $\B$ is well centered in the sense of Section 5. Let $A'$ be the bisector passing through the center $(h,k)$ of $\B$, and let $A$ be the line in ${\mathbb{B}}$ for which $A,{A'}$ is a $\B$-pair. Since the bisector field ${\mathbb{B}}$ is cubic, $\B$ contains only moving bisectors and so   $A$ is not parallel to ${A'}$ by Lemma~\ref{zeroes}.    
 Since  $\k$ is algebraically closed, every element in~$\k$ is a square in $\k$, and so by  Lemma~\ref{preserve} we may assume $\B$ is in standard form, 
   $A$ is the line $Y=0$, ${A'}$ is the line $X=0$ and $\B$ has coefficient $\mu=-1$.  Since  the center $(h,k)$ of $\B$ lies on ${A'}$, we have  $h=0$. If $k =0$, then the centroid of $Q$ is the intersection of $A$ and $A'$. However, using  Lemma~\ref{orthogonal}, this implies $\Psi =0$, in which case the reduced dual polynomial $F$ is linear, contrary to the assumption that $\B$ is cubic. Thus $k \ne 0$, and
   after rescaling by a factor of $1/(2k)$, we can assume the center $(h,k)$ of ${\mathbb{B}}$ is $(0,1/2)$.     
 Substituting 
  $h=0$, $k =1/2$ and $\mu =-1$ into the equation for the boundary from Lemma~\ref{tedious} 
   yields the curve $\Delta$ in the present theorem.   

   We next find the singularities of the curve $\Delta$. 
   Since $\k$ is algebraically closed there is $\theta \in \k$ such that $16\theta^2=27$. 
   A calculation shows there are polynomials $f_1,f_2 \in \k[X,Y]$ 
   for which each term of $f_1$ and $f_2$ has degree at least $2$ and  
\begin{eqnarray*}
\Delta_1(X,Y)&:=&\Delta\left(X-\frac{4}{3}\theta Y+\theta,Y-\frac{1}{4}\right) \: = \: Xf_1(X,Y)+\frac{27}{4}X^2
+16Y^3(Y-1) \\
\Delta_2(X,Y)&:=&\Delta\left(X+\frac{4}{3}\theta Y-\theta,Y-\frac{1}{4}\right) \: =\: Xf_2(X,Y)+\frac{27}{4}X^2
+16Y^3(Y-1) \\
\Delta_3(X,Y)&:=& \Delta\left(X,Y+2\right) \: =\: 
X^4+2X^2Y^2+18X^2Y+27X^2+(Y+2)Y^3
\end{eqnarray*}

Suppose the characteristic of $\k$ is not $3$.  
At the origin $p=(0,0)$ the curves $\Delta_1,\Delta_2,\Delta_3$ each have a singularity of order $2$ and  unique tangent line $X=0$. Also, with ${\mathcal{O}}_p$ the local ring at $p$, we have   ${\mathcal{O}}_p/(\Delta_i,X) = {\mathcal{O}}_p/(Y^3,X)$ and so the intersection multiplicity of the curve $\Delta_i$ and line $X=0$ at $p$ is $3$.  Therefore, the origin is an ordinary cusp in the sense of \cite[Exercise~3.22, p.~82]{Fulton} for each of the curves $\Delta_i$.
Each $\Delta_i$ is the image of $\Delta$ under an affine transformation, and examining these transformations   shows $\Delta$ has cusps at  $(\theta,-{1}/{4})$,  $(-\theta,-{1}/{4})$ and $(0,2)$.
Following \cite[Corollary 1, Section 8.3]{Fulton} or \cite[Lemma~3.24, p.~55]{Hir}, 
the genus $g$ of a plane curve of degree $d$ satisfies $$0 \leq g \leq (1/2)(d-1)(d-2)-\sum_q (1/2)m_q(m_q-1),$$ where $q$ ranges over the points on the curve and $m_q$ is the order of the point $q$ on the curve. Thus the genus $g$ of  $\Delta$ satisfies $0 \leq g \leq 3-(1+1+1)=0$, which proves that the three singularities listed are the only singularities of $\Delta$, these singularities are ordinary cusps and   $\Delta$ is a rational curve. 


If char $\k = 3$, then   $\theta=0$ and $-({4}^{-1}) =2$ and so all three singularities listed above coincide with 
 $(0,2)$, and $\Delta_1=\Delta_2=\Delta_3$ (when the term $(-4/3)\theta$ is replaced by $0$).  Working with $\Delta_3$ and reducing mod $3$ yields $\Delta_3(X,Y) = 
 X^4+2X^2Y^2+Y^4+2Y^3$
Thus $Y=0$ is the unique tangent line to the curve $\Delta_3$ at the origin, and the origin has order $3$ on this curve. This tangent line has multiplicity~$4$ at the origin. Consequently, the singularity $(0,2)$ is a higher-order cusp. The genus $g$ of $\Delta$ satisfies $0 \leq g \leq 3 - 3=0$ so $\Delta$ is a rational curve having a unique singularity, which is a higher-order cusp. 
  \end{proof} 

Removing the restriction to an algebraically closed field, we have

 \begin{corollary} \label{Delta corollary} If $\B$ is a  cubic bisector field, then the boundary of $\B$ is a rational quartic curve that 
has at most three singularities, each of which is a cusp. 
\end{corollary}

\begin{proof}  Extending to the algebraic closure $\overline{\k}$ of $\k$,   Corollary~\ref{Bis cor}  implies the extension 
 $\overline{\B}$  of $\B$ to the plane $\overline{\k} \times \overline{\k}$ is a cubic bisector field over $\overline{\k}$.  Applying Theorem~\ref{Delta theorem} to ${\overline{\B}}$,  the boundary of $\overline{\B}$ is a rational quartic curve that 
has at most three singularities, each of which is a cusp. The boundary of $\B$ is the set of  $\k$-rational points on the boundary of ${\overline{\B}}$, and so the corollary follows.
\end{proof}


 \section{The boundary as an envelope}
 
 We show in this section that the
 moving  bisectors from Definition~\ref{moving} are precisely the lines 
  tangent to the boundary of the bisector  field. The results of the last section allow us to work with bisector fields whose boundaries have tractable explicit equations. Because of this, some of the results in this section, such as Lemma~\ref{Pardini lemma},  can alternatively be verified by brute force using  computer algebra software, but we have opted instead for a more conceptual approach that uses   duality for plane curves, with the caveat that since the base field $\k$ may have positive characteristic, the use of duality requires a little more care. 
 

 For a polynomial $G(X_1,\ldots,X_n) \in \k[X_1,\ldots,X_n]$, we denote the partial derivative of $G$ with respect to $X_i$ as $G_{X_i}.$
   If $G \in \k[X,Y,Z]$ is a nonzero homogeneous   polynomial, the zero set defined by $G$ is denoted  $${\bf V}(G)=\{[x:y:z] \in \P^2(\k):G(x,y,z)=0\}.$$  
   
   Let $\B$ be a nonlinear bisector field, and let $F$ be its   reduced dual polynomial, as in Definition~\ref{dual rational function def}.
   Let  $D(X,Y,Z)$    be the  discriminant of   $$ T \phi(T,U)X  -U\phi(T,U)Y + \psi(T,U)Z$$ 
viewed as a polynomial with coefficients in $\k[X,Y,Z]$. Then $\Delta(X,Y) = D(X,Y,1)$ and $D$ is the homogenization of $\Delta$. Thus $D(x,y,z) = 0$ if and only if $$T \phi(T,U)x  -U\phi(T,U)y + \psi(T,U)z$$ has a repeated root in $\P^1(\overline{\k})$. 
Define {rational maps}  at the nonsingular points of $F$ and $D$ by
\begin{eqnarray*}
&f:&{\bf V}(F) \rightarrow \P^2(\k): [t:u:v] \mapsto  [F_T(t,u,v):-F_U(t,u,v):F_V(t,u,v)] \\
&d:&{\bf V}(D) \rightarrow \P^2(\k): [x:y:z] \mapsto  [D_X(x,y,z):D_Y(x,y,z):D_Z(x,y,z)].
\end{eqnarray*} 
Denote by ${\bf V}(F)^*$ the Zariski closure of the image of $f$ in $\P^2(\k)$ and by    ${\bf V}(D)^*$ the Zariski closure of the image of $d$ in $\P^2(\k)$. The next lemma verifies that $f$ and $d$ define a duality. 


 
 

 \begin{lemma} \label{Pardini lemma} \label{dual of F} If $\k$ is algebraically closed and $\B$ is a nonlinear bisector field, then 
 \begin{center} ${\bf V}(F)^* = {\bf V}(D)$, ${\bf V}(D)^* = {\bf V}(F)$ and 
 $d \circ f =1_{{\bf V}(F)}$. 
\end{center}

 \end{lemma} 
  
 
 \begin{proof} 
 We first prove ${\bf V}(F)^* \subseteq {\bf V}(D)$.  
Let $[t:u:v] \in {\bf V}(F)$. Then $\psi(t,u)=v\phi(t,u)$ and so $\phi(t,u) \ne 0$ since $\phi$ and $\psi$ do not share a common zero in $\P^1(\k)$.    Let $$x = F_T(t,u,v) \:\:\:\:\:\: y = -F_U(t,u,v) \:\:\:\:\:\:  z=F_V(t,u,v).$$
 We show $D(x,y,z) =0.$  To prove this, since 
 $D(x,y,z)$ is the discriminant of $$P(T,U) :=T \phi(T,U)x  -U\phi(T,U)y + \psi(T,U)z,$$ 
 it suffices to show 
  $[t:u]$ is a repeated root of   $P(T,U)$.  Let $n = \deg F$. By Euler's Formula, $nF = TF_T+UF_U+VF_V$.  Using this and the fact that   $v\phi(t,u)=\psi(t,u)$, we have  $$P(t,u) =\phi(t,u)( tx -uy+vz) = n\phi(t,u)F(t,u,v) =0,$$ and so $[t:u]$ is a root of $P$, which implies $tx-uy=-vz$ since $\phi(t,u) \ne 0$.  
 To see that this root is repeated, in the case in which $ u \ne 0$ it is enough to show $P_T(t,u) =0$, and in the case that $t \ne 0$ that $P_U(t,u) =0$.  We show the former and omit the analagous proof for the latter. 
 Observe  \begin{eqnarray*} 
 P_T(t,u) &=& xt\phi_T(t,u) +x\phi(t,u) -yu\phi_T(t,u) +z\psi_T(t,u) \\
 & = & \phi_T(t,u)(xt-yu) + x\phi(t,u) +z\psi_T(t,u) \\
 &=& \phi_T(t,u)(-zv)+ x\phi(t,u) +z\psi_T(t,u) \\ & = & z(\psi_T(t,u)-v\phi_T(t,u)) +x\phi(t,u) \\
 & = & -zF_T(t,u,v) +x\phi(t,u) \:\: = \:\: -zx +x\phi(t,u) \:\: = \:\: 0
 %
%
  \end{eqnarray*}
  (The last equality follows from the fact that   $z = F_V(t,u,v) = \phi(t,u)$.)
This proves that $[t:u]$ is a repeated root of $P(T,U)$, and so $D(x,y,z) = 0$. Thus ${\bf V}(F)^* \subseteq {\bf V}(D)$.
 Since the curve $D$ is by Theorem~\ref{parabola theorem} and Corollary~\ref{Delta corollary} a rational, hence irreducible, curve (it is the projective closure of a rational curve)
%
we obtain ${\bf V}(F)^* = {\bf V}(D)$. 
 
 To prove the remaining assertions that ${\bf V}(D)^* = {\bf V}(F)$   and $d \circ f = 1_{{\bf V}(F)}$, 
  it suffices by \cite[Propositions~1.2 and~1.5]{Pardini} 
 to show $F$ has finitely many flexes. 
 We show in fact that $F$ has at most three flexes. 
 If $\B$ is quadratic, then  ${\bf V}(F)$ is a parabola by Theorem~\ref{parabola theorem}   and hence has no flexes, so we assume $\B$ is cubic.  
 Since $\k$ is algebraically closed, 
we may assume as in the proof of Theorem~\ref{Delta theorem} that $\B$ is in standard form with center $(0,1/2)$ and coefficient~$-1$.   
  By Lemma~\ref{orthogonal},  $\Phi(T,U) = T^2+U^2$ and $\Psi(T,U) = 2T^2U$. Thus $F(T,U,V)$ is a factor of the dual polynomial of $\B$, $G(T,U,V)= 2T^2U - V(T^2+U^2)$.  
 The Hessian of $G$ is the matrix consisting of the second partials $$H=\begin{bmatrix}
 4U-2V & 4T & -2T \\
 4T & -2V & -2U \\
 -2T & -2U & 0 \\
 \end{bmatrix}.$$ 
The flex points and singularities on  the curve $G$ are the points in $\P^2(\k)$ on the intersection of the curve $G$ and   the curve   $\det(H) = 16U(2T^2-U^2)+8V(T^2+U^2)$  \cite[Theorem~1.35, p.~14]{Hir}.  Eliminating $V$ shows   the  points  on this intersection are  the zeroes of $2U(3T^2-U^2) = 0$. 
Two such zeroes are $[0:0:1]$ and $[1:0:0]$.  
 If char $\k = 3$, then 
 these are the only points on the intersection. If  char $\k \ne 3$  and 
  $\theta$ is a  root of $3Z^2-1$, then $[\theta:1:1/2]$ and $[-\theta:1:1/2] $ are the only other points on this intersection. 
  The only singularity on the curve $G$ is $[0:0:1]$, so all the other points listed here are flex points on the curve  $G$.   Thus $G$, and hence $F$ also, has at more three flexes. 
Regardless of whether $\B$ is cubic or quadratic, we have established   the curve   $F$ has  finitely many flexes, and   so the lemma is proved.  \end{proof}


For the next theorem, which is  illustrated by Figures~1(a)--(c), we recall the moving  bisectors   from Definition~\ref{moving}.


\begin{theorem} \label{envelope theorem} The lines tangent to the boundary of a  nonlinear bisector field $\B$ are the moving  bisectors in $\B$. 
%
\end{theorem}


\begin{proof} 
We first prove the theorem in the case in which  $\k$ is algebraically closed. 
 Suppose $\ell:tX-uY+v$ is a line tangent to the boundary  of $\B$ at the point $(x,y)$. To prove that $\ell$ is a moving bisector, it suffices by Proposition~\ref{new moving bisector} to show  $F(t,u,v)=0$.  Now $0=\Delta(x,y) = D(x,y,1)$ and so $[x:y:1] \in {\bf V}(D)$.  
 Let $f$ and $d$ be as in Lemma~\ref{Pardini lemma}. 
 If $[x:y:1]$ is nonsingular, then  $[t:u:v] = d([x:y:1])$. By Lemma~\ref{Pardini lemma}, $[t:u:v] \in {{\bf V}(D)}^{*}={\bf V}(F)$,  and so $F(t,u,v)=0$ and $\ell$ is a moving bisector.
 
  It remains to consider the case in which $[x:y:1]$ is singular.  Since in this case $\Delta$ is a singular curve, Theorems~\ref{parabola theorem} and~\ref{Delta theorem} imply $\B$ is a cubic bisector field and  
  $[x:y:1]$ is a cusp on  $D$. Hence $\ell:tX-uY+v$ is the unique tangent line 
  to $D$ at $[x:y:1]$. 
   Since $\k$ is algebraically closed, we may assume  
 as in the proof of Theorem~\ref{Delta theorem} that $\Phi(T,U) = T^2+U^2$ and $\Psi(T,U) = 2T^2U$. Since $\B$ is cubic, $\Phi = \phi$ and $\Psi = \psi$, and hence the reduced dual polynomial is $F(T,U,V) = 2T^2U - V(T^2+U^2)$. 
As shown in the proof of Theorem~\ref{Delta theorem}, the singularity $[x:y:1]$ of $D$ must be either 
 $[0:2:1]$, $ [ \theta:-1/4:1]$, or $ [ -\theta:-1/4:1]$, where $16\theta^2=27$. Using the ideas from the proof of Theorem~\ref{Delta theorem}, we will find the tangent lines at each of these points and show they are moving bisectors.
 
 The argument in the proof of Theorem~\ref{Delta theorem} shows the tangent line to $\Delta$ at $(0,2)$ is the image of the line $X=0$ under the transformation of the plane  $(x,y) \mapsto (x,y+2)$, and hence is again the line $X=0$. This line has dual point $[1:0:0]$, and since $F(1,0,0)=0$, this line is a moving bisector by Proposition~\ref{new moving bisector}.   The tangent line to $\Delta$ at $(\theta,-1/4)$ is the image of the line $X=0$ under the transformation $(x,y) \mapsto (x-(4/3)\theta y+\theta,y-(1/4))$, which is the line $(3/2)X+2\theta Y-\theta=0$. This line has dual point $[3/2:-2\theta:-\theta]$, which is a zero of $F$, and hence this tangent line is also a moving bisector by Proposition~\ref{new moving bisector}.
  Similarly, the tangent line to $\Delta$ at $(-\theta,-1/4)$ is the image of $X=0$ under the transformation $(x,y) \mapsto (x+(4/3)\theta y -\theta,y-(1/4))$, which is the line $(3/2)X-2\theta Y+\theta=0$. 
The dual point $[3/2:2\theta:\theta]$ of this line is on $F$, so this line is also a moving bisector. 
  This proves that  the lines tangent to the boundary of~$\B$ are moving bisectors in $\B$.  

Conversely,   still under the assumption that $\k$ is algebraically closed, suppose the line $\ell:tX-uY+v$ is a moving bisector. We show $\ell$ is tangent to the boundary of $\B$.  By Proposition~\ref{new moving bisector},  $F(t,u,v) =0$, and so  $[t:u:v]$ is a nonsingular point on the curve $F$ since the only singularity of this curve is $[0:0:1]$. Thus the rational map $f$ from Lemma~\ref{Pardini lemma} is defined at $[t:u:v]$. Let $[x:y:z] =f([t:u:v])$. 
 We will show $tX-uY+vZ=0$ is tangent to  the curve $D$ at $[x:y:z]$.


 Let 
  $t'X-u'Y+v'Z=0$ be the line tangent to the curve $D$  at $[x:y:z]$.  
%
 If $[x:y:z]$ is nonsingular, 
 then $d([x:y:z]) = [t':u':v']$, so Lemma~\ref{Pardini lemma} and the fact that  $f([t:u:v]) = [x:y:z]$ implies 
 $$[t:u:v]=d(f([t:u:v])) = d([x:y:z]) = [t':u':v'].$$ Therefore,
the line $tX-uY+v=0$ is the same as the line $t'X-u'Y+v=0$, and so  
  in the case in which $[x:y:z]$ is nonsingular, the line $tX-uY+vZ=0$ is tangent to the curve $D$ at $[x:y:z].$

 Now suppose $[x:y:z] =f([t:u:v]) $ is a singular point on the curve $D$.
 We show that $tX-uY+vZ=0$ is tangent to $[x:y:z]$ also in this case.  
  By Theorem~\ref{parabola theorem} and Corollary~\ref{Delta corollary},  
    $\B$ is cubic since its boundary has a singular point, and so 
   as in the proof of Theorem~\ref{Delta theorem}, the singularities of $D$ are 
   \begin{center}$[0:2:1],$ \: $[\theta:-1/4:1]$, \: $[-\theta:-1/4:1]$, \: where $16\theta^2=27$.
   \end{center}
   Thus $[x:y:z]$ is one of these points, and 
  the bisector $tX-uY+v=0$ goes through this point. As established earlier, there is a unique tangent line to $D$ at each of these points, and as shown at the beginning of this proof, each tangent line to $D$ is a bisector in $\B$.  Thus to prove the claim that $tX-uY+v=0$ is tangent to $D$ at the singularity $[x:y:z]$, it suffices to show each of these three singularities has at most one bisector passing through it. 
  
 We prove this for the singularity $[0:2:1]$ and omit the very similar proofs for the singularities    $[\pm \theta:-1/4:1]$.
  Since the bisector field $\B$ is cubic, all the bisectors in $\B$ are moving bisectors, and so  by Proposition~\ref{new moving bisector}, the only bisectors passing through $(0,2)$ are of the form $t'X-u'Y+v'$, where  
   $-2u'\Phi(t',u')+\Psi(t',u') =0$. Since $\Phi(T,U) = T^2+U^2$ and $\Psi(T,U) = 2T^2U$, it follows that $2u'((t')^2+(u')^2) = 2(t')^2u'$, so that $u' =0$. It cannot be that $t' =0$, so since $0=F(t',u',v') = 2(t')^2u' - v'((t')^2+(u')^2) = -v'(t')^2$, we have $v'=0$. This implies $t'X-u'Y+v'=0$ is the line $X=0$, and so there is only one bisector through the singularity $[2:1:1]$. 
 This proves that
  if $f([t:u:v])=[x:y:z]$ is a singular point on $D$, the     
       line $tX-uY+vZ=0$ is tangent to the curve $D$ at $[x:y:z]$, which verifies the theorem if $\k$ is algebraically closed.   

Now suppose $\k$ is not necessarily algebraically closed. Let $\Upsilon = \psi/\phi \in \k(T,U,V)$.  By what we have established, when extended to the plane $\overline{\k} \times \overline{\k}$, the moving bisectors  $tX-uY+\Upsilon(t,u)=0$, where $[t:u] \in \P^1(\k)$ and $\phi(t,u) \ne 0$,   are  tangent to ${\bf V}(D)$. If $tX-uY+\Upsilon(t,u)=0$ is such a line, then with $(x,y) \in \overline{\k} \times \overline{\k}$ the point of tangency with the curve $\Delta$,  since $f([t:u:v])=[x:y:1]$ and the polynomials that define the map $f$ are in $\k[T,U,V]$, it follows that $(x,y)$ is a $\k$-rational point, and so $tX-uY+\Upsilon(t,u)$ is tangent to the curve $\Delta$ in $\k^2$.  

Conversely, suppose $tX-uY+v=0$ is a line in $ {\overline{\k}} \times  {\overline{\k}}$ tangent to the boundary at a point $(x,y) \in \k^2$. 
By what we have shown in the first part of the proof, $tX-uY+v$ is a moving bisector in $ {\overline{\k}} \times  {\overline{\k}}$. Thus $\phi(t,u) \ne 0$ and $v = \Upsilon(t,u)$, and to complete the proof we need only verify that the coefficients $t$ and $u$ can be chosen from $\k$.  
 If $(x,y)$ is nonsingular, then since $\Delta$ has coefficients in $\k$ and $f([t:u:v])=[x:y:1]$, it follows that the coefficients of $tX-uY+v$ can be chosen from $\k$.  If $(x,y)$ is singular, then over $\overline{\k}$ this point is a cusp. After a translation, we can assume $(x,y)  = (0,0)$ and hence $v =0$. Since this singularity is a cusp, 
  $\Delta(X,Y) = (tX-uY)^e+g(X,Y)$, where $e>1$ and $g$ is a polynomial with $\deg g > e$.  
Since the coefficients of $\Delta$ are in $\k$ and  $t^au^b \in \k$ for all $a,b \geq 0$ with $a+b = e$, it follows that $tu^{-1} = t^e(t^{e-1}u)^{-1} \in \k$ if $u \ne 0$ and similarly $t^{-1}u \in \k$ if $t \ne 0$.  Thus 
 $t$ and $u$ can be chosen in $\k$, which  completes the proof.  
\end{proof} 

  Since every bisector in   a cubic bisector field  is a moving bisector, Theorem~\ref{envelope theorem} implies 
  
  \begin{corollary} \label{cubic tangents} A cubic bisector field   is the set of lines tangent to its boundary.  \qed
  \end{corollary}



  \section{Classification of bisector fields} 
  
  Using the results of the previous sections, we classify bisector fields up to affine equivalence. For linear and quadratic bisector fields, we obtain a classification that is independent of the nature of the field $\k$.  The case of a cubic bisector field is more complicated and depends on the choice of field $\k$. For cubic bisector fields we  obtain classifications when~$\k$ is an algebraically closed field or a real closed field. We
  give some partial results for finite fields, 
and we   leave as an open question the classification of cubic bisector fields over other fields.

\begin{theorem} \label{linear class} A set of lines in $\k^2$ is a linear bisector field  if and only if it is the union of  two pencils of parallel lines and a pencil of non-parallel lines. Up to affine equivalence, there is only linear bisector field over $\k$.  
\end{theorem}

\begin{proof} 
If $\B$ is a linear bisector field, then $\B$ contains two pencils of parallel lines by Theorem~\ref{new} and these two pencils contain all the null bisectors. Since $\B$ is linear, $\phi(T,U) $ is a nonzero constant, say $c \in \k$,  and $\psi(T,U) =aT-bU$ for some $a,b \in \k$ where possibly $a=b=0$. 
 Thus
 the reduced dual polynomial $F = \psi - V \phi$ defines
 a pencil ${\mathcal{P}}$ of lines  $ctX-cuY+\psi(t,u) = t(cX-a)+u(cY-b)$, where $[t:u] \in \P^1(\k)$. Each such line passes through $(a,b)$ and is  a moving bisector in $\B$ by Proposition~\ref{new moving bisector}. Also, any line $\ell: tX-uY+v=0$  in $\B$ that is not in the two pencils of parallel lines satisfies $\psi(t,u) = v\phi(t,u)$ by Theorem~\ref{cubic} and Lemma~\ref{zeroes}, and so $\ell$ is in the pencil ${\mathcal{P}}$.   
 Thus $\B$ is the union of two pencils of parallel lines and another pencil that consists of the moving bisectors.  

In light of Theorem~\ref{new}, to prove the converse it suffices to show that a union $\B$  of two pencils of parallel lines and a third pencil whose midpoint is a point $p$ in $\k^2$ is a bisector field. Let $\ell_1$ be the line in the first pencil of parallel lines that passes through $p$, and let $\ell_2$ be the line through $p$ in the second pencil, so that $\ell_1$ and $\ell_2$ meet at $p$. 
Let $A,A'$ be a pair of distinct parallel lines that share $\ell_1$ as a midline,  let $B,B'$ be a pair of distinct parallel lines that share $\ell_2$ as a midline, and let 
 $Q=ABA'B'$. It  is straightforward to  see that each line in  $\B$ bisects the parallelogram $Q$ and that any line that bisects $Q$ occurs in one of the three pencils that comprise $\B$.  Therefore, 
 $\B = \Bis(Q)$, and so  $\B$ is a bisector field.  Since any parallelogram can be transformed into any other parallelogram in the plane, it follows that any two linear bisector fields are affinely equivalent. 
\end{proof}


See Figure 1(d) for an illustration of Theorem~\ref{linear class} in the case where $\k$ is the field of real numbers. 
 As with linear bisector fields, the classification of quadratic bisector fields  is independent of the nature of the field~$\k$. 
  
   \begin{theorem}  \label{quad class}  
   A set of lines is a 
 quadratic bisector field if and only if it is affinely equivalent to the union of the set of lines  tangent to the curve $Y-X^2=0$ and the pencil of lines parallel to the line $Y=0$. Up to affine equivalence there is only one quadratic bisector field.

\end{theorem}  

\begin{proof} Let $\B$ be a quadratic bisector field. By Theorem~\ref{parabola theorem}, $\B$ is affinely equivalent to the quadratic bisector field whose boundary is the parabola $Y=X^2$ and is such that the pencil of parallel lines of this bisector field consists of lines parallel to $Y=0$. 
By Theorem~\ref{envelope theorem}, this bisector field consists of lines that are tangent to the curve $Y-X^2=0$ or in the pencil of lines parallel to the line $Y=0$. 
The converse follows from Theorems~\ref{parabola theorem} and~\ref{envelope theorem}. 
%
\end{proof}

Figure~1(c)  shows the only quadratic bisector field up to affine equivalence over the reals. 
The classification of cubic bisector fields is more complicated than that of linear and quadratic bisector fields, but at least if $\k$ is algebraically closed, the classification is definitive and essentially already done.

 \begin{theorem} \label{ACF cubic}
 Suppose $\k$ is an algebraically closed field. A set of lines in $\k^2$ is a  cubic bisector field if and only if it is affinely equivalent to the set of lines tangent to the curve
  \begin{eqnarray}  \label{first deltoid}
 X^4+2X^2Y^2+Y^4+10X^2Y-6Y^3-X^2+12Y^2-8Y = 0. 
\end{eqnarray}

\end{theorem} 

\begin{proof}  Apply Theorem~\ref{Delta theorem} and Corollary~\ref{cubic tangents}. 
\end{proof}


    Over a real closed field,  
 the cubic case splits into two  cases, that of a deltoid and an ``infinite cardioid;'' see Figures~1(a) and~(b). 
 

\begin{theorem} \label{real case} Suppose  $\k$ is a real closed field. 
If $\B$ is  a cubic bisector field, then $\B$ is affinely equivalent to either the set of lines tangent to the deltoid in $(\ref{first deltoid})$ or the cardioid  
$$X^4-2X^2Y^2+Y^4-10X^2Y-6Y^3+X^2+12Y^2-8Y=0.$$
\end{theorem}


\begin{proof}  Since  $\k$ is real closed,  every bisector field $\B$ is well centered, and, as in the proof of Theorem~\ref{Delta theorem}, every cubic bisector field  is affinely equivalent to a bisector field in standard form with center $(0,1/2)$.  
 By Theorem~\ref{big question}, 
two  cubic bisector fields $\B_1$ and $\B_2$ in standard form  with coefficients $\mu_1$ and $\mu_2$, respectively,  are affinely equivalent if and only if $\mu_1$ and $\mu_2$ have the same sign. Thus every bisector field is affinely equivalent either to the  cubic bisector field in standard form  with coefficient $1$ and center $(0,1/2)$ or to the   cubic bisector field in standard form with coefficient $-1$ and center $(0,1/2)$.  By Lemma~\ref{tedious}, the boundary in the former case is the cardioid in the theorem; in the latter case it is the deltoid. 
The theorem follows now from Corollary~\ref{cubic tangents}. 
%
%
\end{proof}

We are only able to  give partial results for other types of fields. 

\begin{example} {\it If $\k$ is the field of rational numbers, then $k$ has infinitely many bisector fields up to affine equivalence.} This is because 
Theorem~\ref{big question lemma} implies that if for each prime integer $p$, $\B_p$ is a bisector field in standard form with coefficient $ p$, then for any two distinct primes $p_1$ and $p_2$, the bisector fields $\B_{p_1}$ and $\B_{p_2}$ are not affinely equivalent. 
\end{example} 

The case of finite fields is more complicated, and we give a few observations about such fields in the following examples.

\begin{example} \label{example 1} {\it If $\k$ is a finite field (as always, of odd characteristic), then  up to affine equivalence there are exactly two well-centered cubic bisector fields.}
Let $\B_1$ and $\B_2$ be well-centered bisector fields over $\k$. Without loss of generality, $\B_1$ and $\B_2$ are in standard form. Let $\mu_1$ and $\mu_2$ denote the coefficients of $\B_1$ and $\B_2$, respectively. Since $\B_1$ and $\B_2$ are well centered, Theorem~\ref{big question} implies $\B_1$ and $\B_2$ are affinely equivalent if and only if $\mu_1\mu_2$ is a square in $\k$. Since $\k$ is a finite field, this is the case if and only if $\mu_1$ and $\mu_2$ are both squares or neither is a square. Thus there are two equivalence classes of well-centered cubic bisector fields. 
\end{example} 


\begin{example} \label{example 2}  {\it Suppose $\k$ is a finite field, and let $\B_1$ and $\B_2$ be bisector fields in standard form with coefficients $\mu_1$ and $\mu_2$, respectively. If either {\rm (a)} $3$ is a square in $\k$ and $\mu_1$ and $\mu_2$ are not squares or {\rm (b)}  $3$ is not a square and $\mu_1$ and $\mu_2$ are squares, then 
$\B_1$ and $\B_2$ are well centered and affinely equivalent.}
Let $(h_i,k_i)$ be the center of $\B_i$.  
Lemma~\ref{well lemma} implies  $\B_i$ is well centered if and only if the homogeneous polynomial $f_i(T,U) =   h_iT^3 + 3k_iT^2U + 3h_i\mu_i T U^2
    +k_i \mu_i U^3
$ has a zero in $\P^1(\k)$; if and only if $f_i$ is reducible over $\k$.  
The polynomial $f_i(T,U)$ has a zero in $\P^1(\k)$ if and only if the cubic $f_i(T,1)$ or the cubic $f_i(1,U)$ has a zero in $\k$; if and only if $f_i(T,1)$ or $f_i(1,U)$ is reducible. 
We claim first  $f_i(T,1)$ is reducible over $\k$, and to show this it suffices by a theorem of Dickson \cite[Theorem~1]{Dickson} to prove the discriminant of $f$ is not a square in~$\k$. The discriminant of $f_i(T,1)$  is $-108\mu_i(h^2\mu_i-k^2)^2$ and hence is a square if and only if $3\mu_i$ is a square. Since $\k$ is a finite field, conditions (a) or (b) guarantee $3 \mu_i$ is not a square for $i=1,2$. Thus $\B_1$ and $\B_2$ are well centered. Since in case (a) or (b), $\mu_1\mu_2$ is a square, Theorem~\ref{big question} implies $B_1$ and $\B_2$ are affinely equivalent. 
\end{example}


The linear and quadratic bisector fields are classified in Theorems~\ref{linear class} and~\ref{quad class}.  
Theorems~\ref{ACF cubic} and~\ref{real case}  classify the cubic bisector fields  over algebraically closed fields and real closed fields, and the preceding examples give partial information for a potential classification over finite fields.  This raises the following question.

\begin{question} \label{question 1}
  Given a field $\k$ of characteristic $\ne 2$, 
what are  the affine equivalence classes of the cubic bisector fields over $\k$? 
\end{question}

Up to affine equivalence we can restrict to cubic  bisector fields in standard form. These bisector fields are entirely determined by their center $(h,k)$ and coefficient $\mu$. Question~\ref{question 1} therefore can be interpreted  as seeking an equivalence relation on triples $(h,k,\mu)$ that encodes the affine equivalence of the corresponding bisector fields. 
 In light of Theorem~\ref{cubic tangents}, another way to put this question is: 

\begin{question} \label{question 2}
  Given a field $\k$ of characteristic $\ne 2$, 
what are  the affine equivalence classes of the  system of 
curves $\Delta$ in Lemma~\ref{tedious}, where $h,k,\mu \in \k$ and $\mu \ne 0$.  
\end{question}

It is straightforward to see that every cubic bisector field is affinely equivalent to a bisector field in standard form with center $(0,1/2)$ or $(1,1)$, and so   the focus in this question can be placed   on $\mu$.

\end{document}